\documentclass[9pt]{article}
\usepackage[a4paper, total={6.2in, 10in}]{geometry}
\usepackage{amssymb, amsmath, amsthm, amstext, amscd, amsgen, amsfonts}
\usepackage{latexsym, wasysym, graphics, mathrsfs}
\usepackage{bm, bbm, longtable}
\usepackage{euscript, enumerate, enumitem, calc, verbatim}
\usepackage{microtype, colonequals} 
\usepackage{tikz}
\usetikzlibrary{cd}
\usetikzlibrary{decorations.pathmorphing}
\usepackage[all]{xy}
\usepackage{color}
\usepackage[colorlinks, pagebackref]{hyperref}
\hypersetup{
	colorlinks=true,
	citecolor=blue,
	linkcolor=blue,
	urlcolor=blue}
\usepackage{url} 
\usepackage[nottoc,numbib]{tocbibind}

\makeatletter
\def\Url@twoslashes{\mathchar`\/\@ifnextchar/{\kern-.2em}{}}
\g@addto@macro\UrlSpecials{\do\/{\Url@twoslashes}}
\makeatother

\newtheorem{theorem}{Theorem}[section]
\newtheorem{lemma}[theorem]{Lemma}
\newtheorem{prop}[theorem]{Proposition}

\newtheorem{sub-lemma}[theorem]{Sub-lemma}

\theoremstyle{definition}

\newtheorem{remark}[theorem]{Remark}

\newcommand{\suchthat}{\;\ifnum\currentgrouptype=16 \middle\fi|\;}
\newcommand{\mycomment}[1]{}

\DeclareMathOperator*{\im}{\mathrm{im}}
\DeclareMathOperator*{\diag}{\mathrm{diag}}
\DeclareMathOperator*{\Lie}{\mathrm{Lie}}

\title{Algebraicity of ratios of special $L$-values for $\mathrm{GL}(n)$}
\author{Ankit Rai \and Gunja Sachdeva}
\date{}

\AtEndDocument{
	
	\bigskip{\footnotesize%
		Ankit Rai \par
		\textsc{Chennai Mathematical Institute, H1 SIPCOT IT Park, Siruseri, Kelambakkam, India, 603103.}\par
		\textit{E-mail address} : \texttt{ankitr@cmi.ac.in}.}
	
	\bigskip{\footnotesize%
		Gunja Sachdeva \par
		\textsc{Department of Mathematics, BITS Pilani, K K Birla Goa Campus, Zuarinagar, Goa 403726, India.}\par
		\textit{E-mail address} : \texttt{gunjas@goa.bits-pilani.ac.in}.}
}

\begin{document}
	
	\maketitle

	\begin{abstract}
		\noindent
		We prove, under certain assumptions, the algebraicity of the ratio $L(m, \Pi \times \chi)/L(m, \Pi \times \chi')$, where $\Pi$ is a cuspidal automorphic cohomological unitary representation of $\mathrm{GL}_n(\mathbb{A}_\mathbb{Q})$, and $\chi$, $\chi'$ are finite-order Hecke characters such that $\chi_{\infty} = \chi'_{\infty} = \mathrm{sgn}^{r}$, and $m, r$ are specific positive integers which depend only on $\Pi_{\infty}$.
		The methods in this article generalize those in the work of Mahnkopf [{\em Cohomology of arithmetic groups, parabolic subgroups and the special values of $L$-functions of $\mathrm{GL}(n)$}, J. Inst. Math. Jussieu, 4 (2005)].
	\end{abstract}


	\tableofcontents

	\section{Introduction}
	The investigation into the algebraic properties of critical values of $L$-functions has a rich history, notably concerning the $L$-functions associated with holomorphic modular forms, a field pioneered by Manin and Shimura in the 1960s.
	In 1979, Deligne \cite{Deligne-L-values} formulated conjectures regarding the special values of $L$-functions associated with motives. Given the close relationship between automorphic representations and motives, it is natural to anticipate similar conjectures concerning the algebraic nature of special values of automorphic $L$-functions.
	This article provides partial confirmation of Deligne's conjecture within the realm of automorphic representations for the specific case under consideration. To present the main theorems of this article, we now introduce the necessary terminology.\\
	
	\noindent
	Let $G_n$ denote the algebraic group $\mathrm{GL}(n)$ defined over $\mathbb{Q}$. Let $T_n \subset B_n \subset G_n$ respectively denote the maximal torus given by the diagonal matrices and the Borel subgroup given by the upper triangular matrices.
	The results in this article are already known in the case when $n$ is even, due to Mahnkopf \cite[p. 557, Remark]{Mahnkopf-05}.
	We proceed, therefore, with the assumption that $n$ is odd.
	Consider the natural identification $\mathbb{Z}^n \cong X^*(T_n)$ given by $(a_1, \dots, a_n) \mapsto \big(\diag (t_1, \dots, t_n) \mapsto t^{a_1}_{1}\cdots t^{a_n}_n \big)$. Let $\mu = (\mu_1, \dots, \mu_{\lfloor n/2 \rfloor}, 0, -\mu_{\lfloor n/2 \rfloor}, \dots, -\mu_1)$ be a character of $T_n$ and $\lambda = (\mu_1, \dots, \mu_{\lfloor n/2 \rfloor}, -\mu_{\lfloor n/2 \rfloor}, \dots, -\mu_1)$ be a character of $T_{n-1}$.
	Assume that $\mu$ is a dominant character. Let $E_{\mu}$ and $E_{\lambda}$ be the finite-dimensional irreducible representations with highest weights $\mu$ and $\lambda$ respectively.
	Let $P_n \subset G_n$ denote the standard parabolic subgroup of $G_n$ of type $(2, \dots, 2, 1)$, and let $P_{n-1} \subset G_{n-1}$ denote the standard parabolic subgroup of type $(2, \dots, 2, 1, 1)$.\\
	
	\noindent
	Let $\mathbb{A}$ be the ring of adeles of $\mathbb{Q}$, and $\mathbb{A}_f$ be the ring of finite adeles. Consider a regular algebraic cuspidal automorphic representation $\Pi$ of $G_n(\mathbb{A})$, assumed to be cohomological with respect to a self-dual representation $E_{\mu}$ with highest weight $\mu$.
	Any (irreducible) automorphic representation admits a factorization as a tensor product $\Pi = \Pi_f \otimes \Pi_{\infty}$, where $\Pi_f$ (resp. $\Pi_{\infty}$) is a representation of $G_n(\mathbb{A}_f)$ (resp. $G_n(\mathbb{R})$). Let $\epsilon(\Pi_{\infty}) = \mathrm{sgn}^{(n-1)/2}$.
	Given two distinct finite-order characters $\chi_{n-2}$ and $\chi_{n-1}$ such that $\chi_{i \infty} = \epsilon(\Pi_{\infty})\mathrm{sgn}$ for $i =n-2, n-1$, define $\chi = \chi_{n-2}|\cdot|^{\frac{2\mu_{\lfloor n/2 \rfloor}+1}{2}} \otimes \chi_{n-1}|\cdot|^{-\frac{2\mu_{\lfloor n/2 \rfloor}+1}{2}}$ as the one-dimensional representation of the torus $diag(1,\dots, 1,\ast, \ast)$.
	Recall that $P_{n-1}$ is a standard parabolic subgroup with a Levi $M_{P_{n-1}} \cong \mathrm{GL}(2) \times \dots \times \mathrm{GL}(2) \times \mathbb{G}_{m} \times \mathbb{G}_{m}$.
	Define $M'_{P_{n-1}} \subset M_{P_{n-1}}$ to be the subgroup consisting of the product of the copies of $\mathrm{GL}(2)$ in $M_{P_{n-1}}$.
	Let $\pi = \pi_1 \otimes \dots \otimes \pi_{\lfloor (n-3)/2 \rfloor}$ be a unitary cuspidal automorphic representation of $M'_{P_{n-1}}(\mathbb{A})$ such that $\pi_{\infty} \cong D_{2\mu_1+n-2} \otimes \dots \otimes D_{2\mu_i + n-2i} \otimes \dots \otimes D_{2\mu_{\lfloor n/2 -1\rfloor} + 3}$. Consider the representation
	\[
	\Sigma \colonequals \mathrm{Ind}^{\mathrm{GL}(n-1, \mathbb{A})}_{P_{n-1}(\mathbb{A})}\left(\pi \otimes \chi \right)
	\]
	of $G_{n-1}(\mathbb{A})$, where $\mathrm{Ind}$ denotes the normalized parabolic induction. The choice of $\pi$ and $\chi$ is made to ensure that $\Sigma_{\infty}$ is cohomological with respect to the finite dimensional-representation $E_{\lambda}$ (See Proposition \ref{prop-coho}).\\
	
	\noindent
	A preliminary aim of this article is to investigate the algebraicity properties of the $L$-value $L(\frac{1}{2}, \Pi \times \Sigma)$ for the Rankin--Selberg $L$-function $L(s, \Pi \times \Sigma)$.
	The choice of $s = \frac{1}{2}$ is forced on us as it is the only critical point of the $L$-function $L(s, \Pi \times \Sigma)$ (See \S \ref{subsec:critical-sets-L-function}).
	One of the key ideas implemented in this article, extensively studied by many authors, including Mahnkopf \cite{Mahnkopf-98, Mahnkopf-05}, Raghuram \cite{Raghuram-12, Raghuram-16}, Raghuram-Sachdeva \cite{Raghuram-Sachdeva-17}, and Grobner \cite{Grobner-Monatsh} among others, is that the $L$-value $L\left( \frac{1}{2}, \Pi \times \Sigma \right)$ admits a cohomological/topological interpretation.
	Before proceeding, we first lay down the assumptions to be used in the rest of this article. With $\pi$ and $\chi$ as above, we assume that the $\pi_i$'s are distinct unitary cuspidal automorphic representations, and $\chi_{n-2} \neq \chi_{n-1}$.
	We say that $\pi$ satisfies $\textbf{Assumption}(\mu, \eta_1, \eta_2)$, for two distinct finite-order Hecke characters $\eta_1$ and $\eta_2$ such that $\eta_{1\infty} = \eta_{2\infty} = \epsilon(\Pi_{\infty})\mathrm{sgn}$, if the following holds
	
	\begin{equation} \label{eqn:nonvanishing-central-L-values}
		L\left( \frac{2\mu_{\lfloor n/2 \rfloor}+1}{2}, \pi^{\vee}_k \times \eta_1 \right) L\left( \frac{2\mu_{\lfloor n/2 \rfloor}+1}{2}, \pi_k \times \eta^{-1}_2 \right) \neq 0 \quad \forall \quad 1 \leq k \leq (n-3)/2.
	\end{equation}

	\noindent
	The nonvanishing of twists of $L$-values for cuspidal automorphic representations is a topic of great interest; see for instance \cite{Rohrlich-89} for certain results in this direction.
	Typically, nonvanishing results fix $\pi$ and prove the existence of several characters $\chi$ such that $L(\frac{1}{2}, \pi \times \chi)$ is non-zero.
	However, our assumptions in the case when $\mu_{\lfloor n/2 \rfloor} = 0$ require us to fix two characters and construct representations $\pi$ such that the central $L$-values in the equation \eqref{eqn:nonvanishing-central-L-values} do not vanish.
	A few results on simultaneous nonvanishing of $L$-values can be found in \cite{Michel-Vanderkam, Dmitrov-Januszewski-Raghuram}.
	We now state the first theorem of this article. This theorem will be applied to demonstrate the central result of this article concerning the algebraicity of the ratios of the twists of special $L$-values of $G_n = \mathrm{GL}(n)$.

	\begin{theorem} \label{thm:main-theorem}
		Let $\Pi$ and $\Sigma$ be as above. Assume that $\pi$ satisfies $\mathbf{Assumption}(\mu, \chi_{n-2}, \chi_{n-1})$. Then there exists a non zero complex number $p(\Pi_\infty, \Sigma_{\infty})$ such that
		\[
		L(\tfrac{1}{2}, \Pi \times \Sigma) \sim_{~\overline{\mathbb{Q}}} p(\Pi_f)p(\Sigma_f) p(\Pi_{\infty}, \Sigma_{\infty})
		\]
		where $\sim_{~\overline{\mathbb{Q}}}$ represents equality up to an element of the algebraic closure $\overline{\mathbb{Q}}$ of $\mathbb{Q}$.
		\textup{(}As $L(\tfrac{1}{2}, \Pi \times \Sigma) \sim_{~\overline{\mathbb{Q}}} L^S(\tfrac{1}{2}, \Pi \times \Sigma)$, we can consider the partial $L$-function $L^S$ in which the Euler product is taken over places outside $S$ containing the places of ramification of $\Pi$ or of $\Sigma$.\textup{)}
	\end{theorem}
	
	\noindent
	The quantities $p(\Pi_f)$ and $p(\Sigma_f)$ arise from the comparison of the two rational structures on the representations; one arising from the Whittaker model (that is, via Fourier coefficients) and the other obtained by the restriction of the Betti rational structure under the maps in equation \eqref{eqn:whittaker-coho-intertwiner} below.
	These are commonly referred to as the Betti-Whittaker periods. Apriori these periods may be unrelated to those put forth by Deligne, but a relation is conjectured in \cite[\S L]{Blasius} (see also \cite[\S 6]{Raghuram-12}). A few related results can be found in \cite{Grobner-Harris, Grobner-Lin-17}.
	The period $p(\Sigma_f)$ decomposes into a product involving a non-zero complex number $\tilde{p}(\Sigma_f)$ (See \S \ref{subsec:rat-Whittaker-Betti-periods}) and a product of certain $L$-functions for the group $G_{2}$ (denoted by $C$ in \S \ref{subsec:global-whittaker-global-integral}).
	The complex number $p(\Pi_{\infty}, \Sigma_{\infty})$ is obtained by evaluating a pairing of two chosen differential forms at infinity, the latter eventually amounting to a sum of local Rankin--Selberg integrals for certain vectors in $\Pi_{\infty}$ and $\Sigma_{\infty}$ (See \S \ref{subsec:Poincare-pairing-Whittaker-vector}).
	Finally, Theorem \ref{thm:main-theorem} yields the main theorem of this article.
	\begin{theorem} \label{thm:rationality-ratio-L-functions-GL(n)}
		Let $\Pi$ be a unitary cuspidal automorphic representation of $\mathrm{GL}(n, \mathbb{A})$ which is cohomological with respect to the finite dimensional representation $E_{\mu}$.
		Let $\chi$ and $\chi'$ be two distinct finite-order algebraic Hecke characters such that $\chi_{\infty} = \chi'_{\infty} = \epsilon(\Pi_{\infty})\mathrm{sgn}$. Assume that either $n=5$ or $\mu_{\lfloor n/2 \rfloor} \geq 1$, then
		\[
		\frac{L(\mu_{\lfloor n/2 \rfloor} + 1, \Pi \times \chi')}{L(\mu_{\lfloor n/2 \rfloor} + 1, \Pi \times \chi)} \in \overline{\mathbb{Q}}^{\times}.
		\]
	\end{theorem}
	
	\noindent
	Before we delve into the history of the problem and the technical aspects of the current article, it is important to discuss the stated and unstated assumptions. We remark that $\mu_{\lfloor n/2 \rfloor}+1$ is a critical point of the $L$-functions $L(s, \Pi \times \chi)$ and $L(s, \Pi \times \chi')$ (see \S \ref{subsec:critical-sets-L-function}).
	In addition to the assumptions of nonvanishing of certain $L$-values mentioned in the statements of the theorems above, we also require that the $L$-values $L(1, \Pi \times \chi)$ and $L(1, \Pi \times \chi')$ do not vanish in order to deduce Theorem \ref{thm:rationality-ratio-L-functions-GL(n)} from Theorem \ref{thm:main-theorem}.
	Such nonvanishing results are known from the work of Jacquet and Shalika \cite{Sarnak-04}. If $\mu_{\lfloor n/2 \rfloor} \geq 1$, then the nonvanishing hypothesis is known from the general results of Shalika.
	Note that the above theorem provides additional evidence supporting a conjecture of Blasius \cite{Blasius} for regular algebraic cuspidal representations that are cohomological with respect to non-regular weights.\\

	\noindent
	In this manuscript we have chosen $\mathbb{Q}$ as the base field to simplify the representation theory at infinity. This simplification enables a clearer presentation of the methods needed to establish the holomorphy, injectivity, and Galois equivariance of the map ``Eis" (see \S \ref{sec:global-map-Eis-non-zero}) induced by an Eisenstein series construction.
	As pointed out to us by an anonymous referee, the methods of this article should extend to any number field.\\
	
	\noindent
	This article grew out of an attempt to generalize Mahnkopf's approach in \cite{Mahnkopf-05} to treat non-regular coefficient systems in the case where $n$ is odd. In his work, Mahnkopf starts with a coefficient system $E_{\mu}$, and an automorphic representation $\Pi$ that is cohomological with respect to $E_{\mu}$.
	He then constructs a \textit{boundary datum}, in particular a choice of a coefficient system, say $E_{\lambda}$, on $G_{n-1}$ ensuring that the induced representation $\Sigma$ on the smaller group $G_{n-1}$ can be realized in the boundary cohomology of $S_{n-1}$ with the chosen coefficient system $E_{\lambda}$.
	This is elaborated in detail in \cite[\S 4.3.1]{Mahnkopf-05}. It is not always possible to make such a choice with parabolic subgroup of the type $(n-3,1,1)$. Instead, Mahnkopf suggests a generalization by considering a parabolic subgroup of the type $(2, \dots, 2, 1, 1)$ (see \cite[p.635]{Mahnkopf-05}), denoted by $P_{n-1}$ in this article.
	When $\Sigma$ is an isobaric sum of unitary cuspidal automorphic representation over a CM field, Grobner \cite{Grobner-Monatsh}, Grobner-Lin \cite{Grobner-Lin-17}, and Grobner-Sachdeva \cite{Grobner-Sachdeva-20} have obtained general results.
	In contrast to these results, we have treated the case with the base field $\mathbb{Q}$ and non-unitary representation $\Sigma$.\\
	
	\noindent
	Experts will recognize that Theorem \ref{thm:main-theorem} is along the same lines as \cite[Theorem B]{Mahnkopf-05}. However, the theorem in \textit{loc. cit.} also proves the Galois equivariance of the ratio of $L$-values.
	The reason that we have proved a weaker statement is due to the softer methods employed in this article. Specifically, we have avoided explicit computations with Whittaker vectors at the ramified places, and instead relied on the computations of Mahnkopf \cite{Mahnkopf-05} to prove the abstract statements.
	These statements appear here as Lemma \ref{lemma:local-Whittaker-choices-ramified} and Lemma \ref{lemma:intertwiner-rational-whittaker-model} and have been essentially  accomplished in the works of \cite{Mahnkopf-05, Grobner-Harris}. It is perhaps natural to expect that a finer analysis should yield a more precise version of Theorem \ref{thm:rationality-ratio-L-functions-GL(n)}.\\

	\noindent
	We will defer the generalization of our results to general number fields and the refinement discussed in the preceding paragraphs to a later time and provide a few comments on the proof.
	Given a compact open subgroup $K_{f} \subset G_n(\mathbb{A}_f)$ and a maximal compact subgroup $K_{n, \infty} \subset G_n(\mathbb{R})$, define the associated locally symmetric space
	\[
	S_n(K_f) \colonequals G_n(\mathbb{Q}) \backslash G_n(\mathbb{A})/K_f \times K^0_{n, \infty}Z^0_n(\mathbb{R}),
	\] 
	where $Z^0_n(\mathbb{R})$ is the connected component of the center of $\mathrm{GL}(n, \mathbb{R}).$
	Consider the vector space
	\[
	\mathrm{H}^{\bullet}(S_n, E_{\mu}) \colonequals \varinjlim {}_{K_f} \mathrm{H}^{\bullet}(S_n(K_f), E_{\mu}).
	\]
	The complex vector space is defined over $\overline{\mathbb{Q}}$ and admits a $G_n(\mathbb{A}_f) \times \pi_0(G_{n-1}(\mathbb{R}))$-action.
	Moreover, according to \cite[\S 4.4]{Mahnkopf-05} it is possible to put a $\overline{\mathbb{Q}}$-structure on $\mathcal{W}(\Pi_f)$\footnote{The nondegenrate character $\psi$ has been suppressed from the notation for convenience} and $\mathrm{Ind}(\mathcal{W}(\pi) \times \chi)$.
	We have the following $G_n(\mathbb{A}_f)$ and $G_{n-1}(\mathbb{A}_f)$-equivariant maps
	\begin{equation} \label{eqn:whittaker-coho-intertwiner}
		\begin{tikzcd}[row sep = tiny]
			\mathcal{W}(\Pi_f) \arrow[r, hook] & \mathrm{H}^{\lfloor n^2/4 \rfloor}_c(S_n, E_{\mu})  \text{ and }  \mathrm{Ind}(\mathcal{W}(\pi) \times \chi) \arrow[r, hook] & \mathrm{H}^{\lfloor (n-1)^2/4 \rfloor}(S_n, E_{\lambda}),
		\end{tikzcd}
	\end{equation}
	respectively. Comparison of the $\overline{\mathbb{Q}}$-structures on the domain and the source of the two maps in \eqref{eqn:whittaker-coho-intertwiner} via the above intertwining maps defines certain nonzero complex numbers $p(\Pi_f)$ and $\tilde{p}(\Sigma_f)$\footnote{Although the notation suggests that $p(\Sigma_f)$ depends only on the finite part, it also depends on
		the choice of a nonzero class $[\Sigma_{\infty}]^{\epsilon} \in \mathrm{H}^{b_{n-1}}(\mathfrak{g}_{n-1}, K^0_{n-1, \infty}; \Sigma_{\infty} \otimes E_{\lambda})$. But the choice once made is fixed throughout this article, so this notation should not cause any confusion. The same remark applies to $\tilde{p}(\Pi_f)$.
		Li-Liu-Sun \cite{Li-Liu-Sun} have constructed a canonical choice of cohomology class $[\Pi_{\infty}]$. It may be plausible to fix a canonical choice of $[\Sigma_{\infty}]$ via Kostant's theorem and using results of \cite{Li-Liu-Sun}.}, respectively.
	This is by now standard for cuspidal automorphic representations. For induced representations this is discussed in \S \ref{subsec:rat-Whittaker-Betti-periods}. Also, see \cite{Grobner-Monatsh, Grobner-Harris} for related results.\\
	
	\noindent
	As discussed in the paragraph preceding Theorem \ref{thm:main-theorem}, the proof hinges on the fact that $L(\frac{1}{2}, \Pi \times \Sigma)$ admits a topological interpretation, that is, it can be written as the Poincar\'{e} pairing (see \S \ref{subsec:Poincare-pairing-Whittaker-vector}) of a specific cohomology class in $\mathrm{H}^{\bullet}(\widetilde{S}_{n-1}, \mathbb{C})$ defined over $\overline{\mathbb{Q}}$.
	This interpretation is a consequence of the fact that the Rankin--Selberg $L$-function admits an integral representation and that a certain numerical coincidence holds.
	This is explained in detail in \cite[\S 2.5.1]{Raghuram-16} and \cite[\S 2.4.4]{Raghuram-16}, respectively, and also summarized in \S \ref{subsec:Poincare-pairing-Whittaker-vector} here.
	As far as the choice of the cohomology classes are concerned, this is achieved by choosing vectors in each automorphic representations $\Pi, \Sigma$ and using the embedding of the automorphic representation in the cohomology of $S_{n-1}$ (see \eqref{eqn:whittaker-coho-intertwiner}) to finally complete this task.
	The \textit{correct} choice of Whittaker vectors at the ramified places needs care. In the present article, we have made an attempt to formalize and prove an abstract lemma (see Lemma \ref{lemma:local-Whittaker-choices-ramified}) pertaining to the choices of Whittaker vectors at ramified places for $G_n \times G_{n-1}$ Rankin--Selberg $L$-functions.
	This has been essentially accomplished in \cite{Grobner-Harris}.\\
	
	\noindent
	For many purposes, for instance for the algebraicity results on $L$-values (up to an algebraic number), the equality of $L$-function and its integral representation can be relaxed at a few finite places to an equality up to a nonzero algebraic number.
	This approach has been used in the present work as well as in several works cited in the previous paragraphs. See \S \ref{subsec:global-whittaker-global-integral} and the paragraph preceding to it for more details.\\
	
	\noindent
	We remark here that our method is closest to that of Mahnkopf \cite{Mahnkopf-98, Mahnkopf-05}. It is noteworthy that the representation $\Sigma$ is non-unitary in the work of Mahnkopf as well as in the present work which is in contrast with almost all the other works mentioned above.
	As the highest weight of the trivial representation is not regular, our work is not subsumed in the existing body of the literature.\\
	
	\noindent
	In a recent work Chen \cite{Chen-23} has considered the case of non-unitary representations $\Sigma$ induced from parabolic of the form $(n_1,n_2, \dots, n_r)$, where at most one of the $n_i$'s are odd.
	Under certain assumptions on $\Sigma$, Chen proves \cite[Theorem 4.2]{Chen-23} the Galois equivariance of ratio $\frac{L(\frac{1}{2} +m, \Pi \times \Sigma)}{p(\Pi)p(\Sigma)p(\Pi_{\infty}, \Sigma_{\infty})}$, where $\frac{1}{2}+m$ is a critical point, and $p(\Pi), p(\Sigma)$ are periods arising from comparison of rational structures. As a corollary, Chen obtains Deligne's conjecture for symmetric power $L$-functions of modular forms.
	Because of the assumption on $n_i$, the results of the current article does not follow from the results of \cite{Chen-23}.\\

	\noindent
	We end this introduction with a brief history :
	\begin{itemize}
		\item $n = 2$. When $F$ is totally real, which is the case of the classical Hilbert modular form, the rationality result was proved by Shimura \cite[Theorem 4.3]{Shimura-78} and Harder \cite{Harder-83}. For a general number field $F$, it is due to Hida \cite[Theorem I(i)]{Hida-94}.
		
		\item $n = 3$. When $F = \mathbb{Q}$, the result is proved by Mahnkopf in \cite{Mahnkopf-98} by taking $\Sigma$ to be a representation induced from two non-unitary characters. Finally, Mahnkopf's result is generalized by Raghuram-Sachdeva \cite{Raghuram-Sachdeva-17} for any totally real field and by Sachdeva \cite{Sachdeva-20} for any CM field.
		
		\item For a general $n$ and $F = \mathbb{Q}$ : A result stronger than Theorem \ref{thm:rationality-ratio-L-functions-GL(n)} is proved in \cite{Mahnkopf-05} when $\Pi$ is cohomological with respect to a regular representation. It is further remarked that the case when $n$ is even does not require the regularity hypothesis.
	\end{itemize}

	\subsection{Organization of the article}
	The article is organized as follows:
	Section 2 gives a short introduction to the various notions used in this article.
	Section 3.1 contains the necessary results on cohomological representations, in particular, it provides a proof of the fact that $\Sigma_{\infty}$, a reducible representation of  $G_{n-1}(\mathbb{R})$, is cohomological.
	Section 3.2 discusses results on global intertwining operators. These results are used to prove that a certain Eisenstein series is well-defined at a given point, yielding an intertwining map from the representation $\Sigma_f$ to the cohomology of the locally symmetric space associated with $G_{n-1}$.
	Section 4 proves the aforementioned result on Eisenstein series and obtains the Galois equivariance of the map $\mathrm{Eis^{coh}}$.
	Section 5 pins down the choices of Whittaker vectors $w_{\Pi}$ and $w_{\Sigma}$ for $\Pi$  and $\Sigma$ respectively, such that the global Rankin--Selberg integral $\Psi(\frac{1}{2}, w_{\Pi}, w_{\Sigma})$(see \S \ref{subsec:global-whittaker-global-integral}) is the Rankin--Selberg $L$-function up to $\overline{\mathbb{Q}}^{\times}$.
	Section 6 completes the proof of Theorem \ref{thm:main-theorem}.
	Section 7 completes the proof of Theorem \ref{thm:rationality-ratio-L-functions-GL(n)}.

	\subsection{Acknowledgements}
	The present work started as a reading project of the papers of Mahnkopf \cite{Mahnkopf-98, Mahnkopf-05} by the two authors with Dipendra Prasad whom we would like to thank for the discussions and constant encouragement. 
	We thank Laurent Clozel for directing us to \cite[\S 4.6]{Moeglin-00}, Nadir Matringe for confirming and outlining a proof of a version of Lemma 5.1, and Arvind Nair and Ravi Raghunathan for several helpful discussions.
	We thank anonymous referees for their valuable comments to improve the article, and encouragement to generalize the scope of our results. The second named author is partly supported by DST-SERB POWER Grant No. SPG/2022/001738.

	\section{Preliminaries}

	\subsection{Automorphic representations}
	Let $\mathcal{A}(G_n(\mathbb{Q})\backslash G_n(\mathbb{A}))_{\eta}$ denote the space of automorphic forms on $G_n$ which transform as a character $\eta : Z_n(\mathbb{A}) \rightarrow \mathbb{C}^{\times}$. This space comes equipped with the right regular action of $G_n(\mathbb{A})$.
	To an automorphic form $\phi$ on $G_n$ and any parabolic subgroup $M_PU_P = P_{n-1} \subset G_n$, we may associate a function $\phi^{U_P} \in C^{\infty}(G_n(\mathbb{Q})U_P(\mathbb{A})\backslash G_n(\mathbb{A}))$ defined as follows
	\[
	\phi^{U_P}(g) \colonequals \int_{U_P(\mathbb{A})} \phi(ug) du.
	\]
	It is known that the constant term determines many aspects of the theory of automorphic forms. An automorphic form $\phi \in \mathcal{A}(G_n(\mathbb{Q})\backslash G_n(\mathbb{A}))_{\eta}$ is said to be a cusp form if and only if $\lbrace \phi^{U_P} \equiv 0 \suchthat P_{n-1} \subset G_n$ is a parabolic subgroup$\rbrace$.
	The subspace of all cusp forms is denoted by $\mathcal{A}_{cusp}(G_n(\mathbb{Q})\backslash G_n(\mathbb{A}))_{\eta}$ and can be easily seen to be preserved under the action of $G_n(\mathbb{A})$. In fact this representation decomposes with finite multiplicity (with multiplicity one for $G_n$) and each of the irreducible constituent is called a cuspidal automorphic representation.\\
	
	\noindent
	The subspace $\mathcal{A}_{cusp}(G_n(\mathbb{Q})\backslash G_n(\mathbb{A}))_{\eta} \subset \mathcal{A}(G_n(\mathbb{Q})\backslash G_n(\mathbb{A}))_{\eta}$ admits a direct summand and this summand is denoted by $\mathcal{A}_{Eis}(G_n(\mathbb{Q})\backslash G_n(\mathbb{A}))_{\eta}$.
	The subscript ``Eis" indicates that this subspace is generated by cuspidal Eisenstein series and its residues -- a nontrivial fact due to Selberg in the case of general linear groups and Waldspurger for general reductive groups. This subspace has been extensively studied in various contexts, in particular, this work crucially depends on the results of Franke \cite{Franke, FS} and their refinement in \cite{Grobner-residual}.
	We now mention a key construction which forms the building blocks of $\mathcal{A}_{Eis}(G_n(\mathbb{Q})\backslash G_n(\mathbb{A}))_{\eta}$.
	For a parabolic subgroup $U_PM_P = P_{n-1} \subset G_n$, denote by $\mathrm{Ind}^{G_n(\mathbb{A})}_{P_{n-1}(\mathbb{A})}$ the normalized induction functor. For $\sigma$ a cusp automorphic representation of the Levi $M_P(\mathbb{A})$ and $\chi$ a character of the center of the Levi,
	the induced representations $\mathrm{Ind}^{G_n(\mathbb{A})}_{P_{n-1}(\mathbb{A})}(\sigma \otimes \chi^s)$ for $s \in \mathbb{C}$ admits a map to the space of automorphic forms $\mathcal{A}\left(G_n(\mathbb{Q})\backslash G_n(\mathbb{A})\right)$ via Eisenstein series. More precisely the realization of induced representations as functions is a two step process as below
	\[
	\mathrm{Ind}^{G_n(\mathbb{A})}_{P_{n-1}(\mathbb{A})}(\sigma \otimes \chi^s) \hookrightarrow C^{\infty}(P_{n-1}(\mathbb{Q})U_P(\mathbb{A})\backslash G_n(\mathbb{A})) \xrightarrow{\mathrm{Eis}} \mathcal{A}(G_n(\mathbb{Q}) \backslash G_n(\mathbb{A})).
	\]
	It is known the Eisenstein series converges when $s$ lies in a certain half plane, and can be continued meromorphically to all $s \in \mathbb{C}$, thus the above map makes sense for any $s \in \mathbb{C}$ (taking residues if needed which we will not need to consider). We refer the reader to \cite{Borel} for more on automorphic forms and automorphic representations.

	\subsection{Locally symmetric spaces and their cohomology} \label{subsec:locally-symmetric-spaces-and-cohomology}
	For any compact open subgroup $K_f \subset G_n(\mathbb{A}_f)$ and a maximal connected  compact subgroup $K^0_{n, \infty}$ of $G_n(\mathbb{R})$, define the associated locally symmetric spaces to be
	\[
	\tilde{S}_n(K_f) \colonequals G_n(\mathbb{Q}) \backslash G_n(\mathbb{A})/ (K_f K_{n,\infty}^0), \quad S_n(K_f) \colonequals G_n(\mathbb{Q}) \backslash G_n(\mathbb{A})/ (K_f K_{n,\infty}^0 Z_n(\mathbb{R})^0),
	\]
	and call the quotient map $p : \tilde{S}_n(K_f) \rightarrow S_n(K_f)$. There is a $\pi_0(G_n(\mathbb{R})) \times G_n(\mathbb{A}_f)$-module
	\[
	H^\bullet(S_n, E) \colonequals \varinjlim {}_{K_f} \mathrm{H}^{\bullet}(S_n(K_f), E),
	\]
	and the cohomology of $S_n(K_f)$ can be recovered by taking invariants under the action of the group $K_f$. Recall the well-known $G_n(\mathbb{A}_f) \times \pi_0(G_n(\mathbb{R}))$-equivariant isomorphism known as Borel's conjecture and now a theorem due to Franke \cite{Franke}
	\[
	\mathrm{H}^{\bullet}(S_n, E) \simeq \mathrm{H}^{\bullet}(\mathfrak{g}_n, K^{0}_{n,\infty}Z_n(\mathbb{R})^0; \mathcal{A}(G_n(\mathbb{Q})\backslash G_n(\mathbb{A})) \otimes E)(\chi_E).
	\]
	Let $\iota : G_{n-1} \rightarrow G_n$ be the map $g \mapsto \begin{pmatrix} g& \\ &1 \end{pmatrix}$. Then $\iota$ induces a map on locally symmetric spaces of respective groups, which will also be denoted by $\iota$.

	\subsubsection{Cuspidal cohomology}
	The inclusion $\mathcal{A}_{{cusp}}(G_n(\mathbb{Q})\backslash G_n(\mathbb{A})) \hookrightarrow \mathcal{A}(G_n(\mathbb{Q})\backslash G_n(\mathbb{A}))$ of the space of cusp forms in the space of automorphic functions induces, via results of Borel and Franke, an injection in $(\mathfrak{g}_n, K_{n, \infty})$-cohomology. Cuspidal cohomology is defined to be the image
	\[
	\mathrm{H}^{\bullet}_{cusp}(S_n(K_f), E) \colonequals \im\left( \mathrm{H}^{\bullet}(\mathfrak{g}_n, K^{0}_{n,\infty}Z_n(\mathbb{R})^0; \mathcal{A}_{cusp}(G_n(\mathbb{Q})\backslash G_n(\mathbb{A}))^{K_f} \otimes E) \hookrightarrow \mathrm{H}^{\bullet}(S_n(K_f), E) \right).
	\]
	Using the decomposition of the space of cusp forms into a direct sum of cuspidal automorphic representations, we get the following fundamental decomposition of $\mathrm{H}^{\bullet}_{cusp}(S_n, E)$ as $\pi_0(G_n(\mathbb{R})) \times G_n(\mathbb{A}_{f})$-modules
	\[
	\mathrm{H}_{cusp}^{\bullet}(S_n, E) = \bigoplus_{\substack{\Pi \subset \mathcal{A}_{cusp}(G_n(\mathbb{Q})\backslash G_n(\mathbb{A}))}_{\chi_E}} \Pi_f \otimes \mathrm{H}^{\bullet}(\mathfrak{g}_n, K^{0}_{n,\infty}Z_n(\mathbb{R})^0; \Pi_\infty \otimes E).
	\]

	\subsubsection{Eisenstein cohomology}
	The inclusion $\mathcal{A}_{Eis}(G_n(\mathbb{Q})\backslash G_n(\mathbb{A}))_{\chi_E} \hookrightarrow \mathcal{A}(G_n(\mathbb{Q})\backslash G_n(\mathbb{A}))_{\chi_E}$ induces a map in $(\mathfrak{g}_n, K^0_{n, \infty})$-cohomology. Define
	\[
	\mathrm{H}^{\bullet}_{Eis}(S_n(K_f), E) \colonequals \mathrm{im}\left( \mathrm{H}^{\bullet}(\mathfrak{g}_n, K_{n, \infty}Z_n(\mathbb{R})^0; \mathcal{A}_{Eis}(G_n(\mathbb{Q})\backslash G_n(\mathbb{A}))^{K_f} \otimes E)_{\chi_E} \hookrightarrow \mathrm{H}^{\bullet}(S_n(K_f), E) \right).
	\]
	This subspace of cohomology is called the Eisenstein cohomology.
	The decomposition of the space $\mathcal{A}_{Eis}(G_n(\mathbb{Q})\backslash G_n(\mathbb{A}))_{\chi_E}$ according to their cuspidal support allows us to understand the Eisenstein cohomology in terms of simpler objects (See \cite{Franke, FS}).
	Unlike cuspidal cohomology, the classes in Eisenstein cohomology often cannot be represented by forms with compact support.

	\subsubsection{Orientation class and the map $\int$} \label{subsubsec:orientation-class-int}
	Recall that
	\[
	\widetilde{S}_n(K_f) = \underset{x \in \widehat{\mathbb{Z}}^{\times}/\det(K_f)}{\bigsqcup} \Gamma_x\backslash G_n(\mathbb{R})/K^0_{n, \infty},
	\]
	where $\Gamma_x \colonequals \tilde{x}^{-1}K_f\tilde{x} \cap G_n(\mathbb{R})$ for any lift $\tilde{x}$ of $x$ to $G_n(\mathbb{A}_f)$. There is an action of $\pi_0(K_{n, \infty})$ on the set of connected components of $\widetilde{S}_n(K_f)$ given by sending $x$ to $x \cdot h_{\infty}$, where $h_{\infty} \in \pi_0(K_{n, \infty})$.
	Choose a non-vanishing section of the line bundle $\det\left( T^*(G_n(\mathbb{R})/K^0_{n, \infty})\right)$, which is tantamount to a choice of an orientation of the manifold $G_n(\mathbb{R})/K^0_{n, \infty}$.
	This fixes a choice of orientation on each connected component of $\widetilde{S}_n(K_f)$ and we denote the collection of orientations by $[\mathrm{or}_{x}]$.
	The behaviour of these orientation classes under the action of $\pi_0(K_{n, \infty})$ is described as in \cite[\S 5.1.3, Lemma]{Mahnkopf-05}. From here a quick computation says that the action is via $\mathrm{sgn}(h_{\infty})$ when $n$ is even.\\
	
	\noindent
	Once the choices of the orientations is fixed, define, for $\alpha = (\alpha_x)_{x \in \widehat{\mathbb{Z}}^{\times}/\det(K_f)}$, the map $\int \alpha \colonequals \mathrm{vol}(K_f)^{-1} \sum_{x \in \widehat{\mathbb{Z}}^{\times}/\det(K_f)} \left(\alpha_x \cap [\mathrm{or}_x] \right)$.
	This map is compatible with pullback maps induced by the morphism $\widetilde{S}_n(K_f) \rightarrow \widetilde{S}_n(K'_f)$ whenever $K'_f \subset K_f$. Thus the map $\int : \mathrm{H}^{\dim S_n}_c(\widetilde{S}_n, \mathbb{C}) \rightarrow \mathbb{C}$ is well defined and is moreover defined over $\mathbb{Q}$.

	\subsection{Cohomological representations} \label{subsec:coh-rep-recall}
	We recall some notations from the introduction. Let $T_n \subset B_n \subset G_n$ respectively be the maximal torus given by the diagonal matrices and the Borel subgroup given by upper triangular matrices. Recall that we have the following isomorphism $X^*(T_n) \cong \mathbb{Z}^n$ which sends $(a_1, \dots, a_n)$ to $\big(\diag (t_1, \dots, t_n) \mapsto t^{a_1}_{1}\cdots t^{a_n}_n \big)$.
	Let $\mu = (\mu_1, \dots, \mu_{\lfloor n/2 \rfloor}, 0, -\mu_{\lfloor n/2 \rfloor}, \dots, -\mu_1)$ be a character of $T_n$ and let $E_{\mu}$ be the finite dimensional representation of $G_n$ with highest weight $\mu$.
	Define $\lambda = (\mu_1, \dots, \mu_{\lfloor n/2 \rfloor}, -\mu_{\lfloor n/2 \rfloor}, \dots, -\mu_1)$ a character of $T_{n-1}$ and let $E_{\lambda}$ be the finite dimensional representation of $G_{n-1}$ with highest weight $\lambda$.
	The classical branching laws \cite[Ch. 8, Theorem 8.1.1]{Goodman-Wallach} tell us that there exists a unique (up to scalars) nonzero $G_{n-1}$-equivariant intertwining map
	\[
	E_{\mu}|_{G_{n-1}} \otimes E_{\lambda} \rightarrow \mathbb{C},
	\]
	which is furthermore defined over $\mathbb{Q}$. In the above displayed equation the subgroup $G_{n-1}$ is embedded in $G_n$ via the map $\iota$.\\
	
	\noindent
	Let $\tau$ denote an infinite dimensional irreducible representation of the group $G_n(\mathbb{R})$. Then $\tau$ is said to be cohomological if $\mathrm{H}^{\bullet}(\mathfrak{g}_n, K^0_{n, \infty}, \tau \otimes E) \neq 0$ for some $\bullet \in \mathbb{N}$ and a finite dimensional representation $E$.
	Let $D_l$ denote the discrete series representation of $G_2(\mathbb{R})$ with lowest weight $l+1$. Then it follows from \cite{Clozel-motifs} and summarized in \cite[\S 3]{Mahnkopf-05} that a representation $\tau$ of $G_{n}(\mathbb{R})$ is cohomological with respect to $E_{\mu}$ if and only if
	\[
	\epsilon(\tau) \otimes \tau \cong \mathrm{Ind}^{\mathrm{GL}(n, \mathbb{R})}_{P_{(2, \dots, 2, 1)}(\mathbb{R})} \left( D_{2\mu_1+n-1} \otimes \dots \otimes D_{2\mu_{\lfloor n/2 \rfloor} + 2} \otimes \mathrm{Id} \right),
	\]
	where $\epsilon(\tau)$ is either $\mathrm{sgn}$ or trivial representation of $\pi_0(G_{n}(\mathbb{R}))$. Similarly a representation $\tau$ of $G_{n-1}(\mathbb{R})$ is cohomological with respect to $E_{\lambda}$ if and only if
	\[
	\tau \cong \mathrm{Ind}^{\mathrm{GL}(n-1, \mathbb{R})}_{P_{(2, \dots, 2)}(\mathbb{R})}\left( D_{2\mu_1+n-2} \otimes \dots \otimes D_{2\mu_i + n-2i} \otimes \dots \otimes D_{2\mu_{\lfloor n/2 \rfloor} + 1}\right).
	\]

	\subsection{Critical set for certain Rankin--Selberg $L$-functions} \label{subsec:critical-sets-L-function}
	For integers $n,m\geq 1$, let $\pi_1$ (resp. $\pi_2$) be an irreducible admissible representation of $G_n(\mathbb{A})$ (resp. $G_m(\mathbb{A})$). Let $L(s,\pi_1 \times \pi_2)= \prod_{v < \infty} L(s,\pi_{1v} \times \pi_{2v})$ be the Rankin--Selberg $L$-function associated to the pair $(\pi_1, \pi_2)$.
	Following \cite[Prop.-Def.2.3]{Deligne-L-values}, call an integer or a half integer $s_0\in \frac{n-m}{2}+\mathbb{Z}$ to be critical for $L(s,\pi_1 \times \pi_2)$ if both $L(s,\pi_{1\infty} \times \pi_{2\infty})$ and $L(1-s,\pi^\vee_{1\infty} \times \pi^\vee_{2\infty})$ are regular at $s = s_0$.\\

	\noindent
	We have the following equality of $L$-functions
	\[
	L(s, \Pi \times \Sigma) = L(s + \frac{2\mu_{\lfloor n/2 \rfloor}+1}{2}, \Pi \times \chi_{n-2})\cdot L(s - \frac{2\mu_{\lfloor n/2 \rfloor}+1}{2}, \Pi \times \chi_{n-1}) \cdot \prod^{(n-3)/2}_{i=1}L(s, \Pi \times \pi_i).\]
	The critical points for $L$-functions on $G_n$ have been computed in \cite[\S 3.1.5]{Mahnkopf-05}. Using this we get that $s = \frac{1}{2} + m$ is critical for $L(s + \frac{2\mu_{\lfloor n/2 \rfloor}+1}{2}, \Pi \times \chi_{n-2})$ and $L(s - \frac{2\mu_{\lfloor n/2 \rfloor}+1}{2}, \Pi \times \chi_{n-1})$ if and only if $m = 0$, where $\chi_{n-2}, \chi_{n-1}$ are as in the introduction.
	Following the recipe as explained in the proof of \cite[Lemma 2.24]{Raghuram-16}, it can be checked that $s = \frac{1}{2}$ is a critical point for the $L$-function $L(s, \Pi \times \pi_i)$ for any $i=1,2,\dots, \frac{n-3}{2}$.
	Hence $s=\frac{1}{2}$ is the unique critical point for the $L$-function $L(s, \Pi \times \Sigma)$.

	\section{Some results in representation theory} \label{sec:result-in-rep-theory}
	Let $\Pi$ be a cuspidal automorphic representation of $G_n(\mathbb{A})$ which is cohomological with respect to the coefficient system $E_{\mu}$, and $\pi$ a cuspidal automorphic representation of the Levi subgroup of $P_{n-1}(\mathbb{A})$ such that $\pi_{\infty} \cong D_{2\mu_1+n-2} \otimes \dots \otimes D_{2\mu_i + n-2i} \otimes \dots \otimes D_{2\mu_{\lfloor n/2 \rfloor -1} + 3}$.
	Consider two distinct finite-order algebraic Hecke characters $\chi_{n-2}$ and $\chi_{n-1}$ such that $\chi_{i\infty} = \epsilon(\Pi_{\infty})\mathrm{sgn}$ for $i=n-2,n-1$, and define
	\begin{equation} \label{eqn:def-Sigma}
		\Sigma \colonequals \Sigma(\pi, \chi_{n-2}, \chi_{n-1}) \colonequals \mathrm{Ind}^{G_{n-1}(\mathbb{A})}_{P_{n-1}(\mathbb{A})}\left(\pi \otimes \chi_{n-2}|\cdot|^{\frac{2\mu_{\lfloor n/2 \rfloor}+1}{2}} \otimes \chi_{n-1}|\cdot|^{-\frac{2\mu_{\lfloor n/2 \rfloor}+1}{2}} \right),
	\end{equation}
	an induced representation of $G_{n-1}(\mathbb{A})$.
	To keep the notation simple, we will simply write $\Sigma$ and ignore the symbols in the parenthesis when there is no possibility of confusion.
	Recall that we denote the character $\chi_{n-2}|\cdot|^{\frac{2\mu_{\lfloor n/2 \rfloor}+1}{2}} \otimes \chi_{n-1}|\cdot|^{-\frac{2\mu_{\lfloor n/2 \rfloor}+1}{2}}$ by $\chi$.

	\subsection{Representation theory: Local} \label{subsec:Representation theory:Local}
	\subsubsection{$(\mathfrak{g}, K_{\infty})$-cohomology of $\Pi$ and $\Sigma$} \label{subsubsec:coho-real}
	Define
	\begin{equation} \label{eqn:cohomological-part}
		\begin{split}
			\tau_{n-1} &\colonequals \mathrm{Ind}^{\mathrm{GL}(n-1, \mathbb{R})}_{P_{(2, \dots, 2)}(\mathbb{R})} \left( D_{2\mu_1+n-2} \otimes \dots \otimes D_{2\mu_i + n-2i} \otimes \dots \otimes D_{2\mu_{\lfloor n/2 \rfloor} + 1}\right),\\
			\tau_n &\colonequals \Pi_{\infty}.
		\end{split}
	\end{equation}
	The inclusion $D_{2\mu_{\lfloor n/2 \rfloor}+1} \subset \mathrm{Ind}^{G_2(\mathbb{R})}_{B_2(\mathbb{R})}\chi_{\infty}$ induces an inclusion
	\begin{equation} \label{eqn:cohomological-part-inclusion-map}
		\tau_{n-1}\subset \Sigma_{\infty}.
	\end{equation}

	\begin{prop} \label{prop-coho}
		Let $\Pi_{\infty}$ and $\Sigma_{\infty}$ be as above and $b_n = \lfloor n^2/4 \rfloor$. Then,
		\begin{enumerate}
			\item $\mathrm{H}^{b_n}(\mathfrak{g}_{n}, K^0_{n, \infty}Z_n(\mathbb{R})^0; \Pi_{\infty} \otimes E_{\mu})$ is one dimensional and $\pi_0(K_{n, \infty})$ acts via $\epsilon(\Pi_{\infty})$.
			
			\item The map induced in Lie algebra cohomology by the inclusion in \eqref{eqn:cohomological-part-inclusion-map} is a $\pi_0(K_{n-1, \infty})$-equivariant isomorphism in degree $b_{n-1}$. Consequently, we have the following $\pi_0(K_{n-1, \infty})$-equivariant isomorphism
			\[
			\mathrm{H}^{b_{n-1}}(\mathfrak{g}_{n-1}, K^0_{n-1, \infty} Z_{n-1}(\mathbb{R})^0 ; \Sigma_{\infty} \otimes E_{\lambda}) \cong 1 \oplus \mathrm{sgn}
			\]
		\end{enumerate}
	\end{prop}
	
	\begin{proof}
		~
		\begin{enumerate}
			\item See \cite[Lemma 3.14]{Clozel-motifs} for a proof.
			\item Recall the following short exact sequence of $G_{2}(\mathbb{R})$-representations
			\[
			0 \rightarrow D_{2\mu_{\lfloor n/2 \rfloor}+1} \rightarrow \mathrm{Ind}^{\mathrm{GL}(2, \mathbb{R})}_{B_2(\mathbb{R})}(\chi_{\infty}) \rightarrow \mathbb{C}_{\varepsilon} \rightarrow 0,
			\]
			where $\mathbb{C}_{\varepsilon}$ denotes the one dimensional representation of $G_2(\mathbb{R})$ given by $\epsilon(\Pi_{\infty})\mathrm{sgn}$.
			Consider the above short exact sequence as representations of $G_{n-3}(\mathbb{R}) \times G_2(\mathbb{R})$ via projection to second component and tensor with $\pi_{\infty}$ considered as the representation of $G_{n-3}(\mathbb{R}) \times G_2(\mathbb{R})$ via projection onto the first component to get
			\[
			0 \rightarrow \pi_{\infty} \otimes D_{2\mu_{\lfloor n/2 \rfloor}+1} \rightarrow \pi_{\infty} \otimes \mathrm{Ind}^{\mathrm{GL}(2, \mathbb{R})}_{B_2(\mathbb{R})}(\chi_{\infty}) \rightarrow \pi_{\infty} \otimes \mathbb{C}_{\varepsilon} \rightarrow 0.
			\]
			Applying the induction functor $\mathrm{Ind}^{(\mathfrak{g}_{n-1}, K^0_{n-1, \infty})}_{(\mathfrak{p}_{(2,\dots, 2)}, K^0_{P_{(2, \dots, 2)},\infty})}$ (here $K^0_{P_{(2, \dots, 2)},\infty} = K^0_{n-1, \infty} \cap P_{(2, \dots, 2)}(\mathbb{R})$), we get an exact sequence of $(\mathfrak{g}_{n-1}, K^0_{n-1,\infty})$-modules. Using the long exact sequence in cohomology and noting that $\mathrm{H}^{i}(\mathfrak{g}_{n-1}, K^0_{n-1,\infty}; \mathrm{Ind}(\pi_{\infty} \otimes \mathbb{C}_{\varepsilon})) = 0$ for all $i$ and thus we get the result.
			The last assertion follows by observing that the induced representation $\mathrm{Ind}(\pi_{\infty} \otimes \mathbb{C}_{\varepsilon})$ is irreducible and we know from the classification of cohomological representations of $G_{n-1}(\mathbb{R})$ that it is not cohomological.
		\end{enumerate}
	\end{proof}

	\subsubsection{Non-vanishing of Rankin--Selberg integral at infinity} \label{subsubsec:critical-point}
	Since $\frac{1}{2}$ is a critical value for $L_{\infty}(s, \tau_n \times \tau_{n-1})$ \cite{Raghuram-16}, the Rankin--Selberg integral gives a well-defined $G_{n-1}(\mathbb{R})$-equivariant nonzero pairing
	\[
	\tau_{n} \;\widehat{\otimes}\; \tau_{n-1} \rightarrow \mathbb{C}.
	\]
	The notation $\widehat{\otimes}$ denotes the completed projective tensor product in the above equation. It follows from \cite[Theorem 1.2(ii)]{Cogdell-Shapiro} that the Rankin--Selberg integral $\Psi_{\infty}(\frac{1}{2}, \cdot, \cdot)$ is nonzero on $\mathcal{W}(\Pi_{\infty}, \psi) \times \mathcal{W}(\Sigma_{\infty}, \psi^{-1})$.
	Binyong Sun proves a more precise result by showing that the Rankin--Selberg integral is nonzero when restricted to the unique minimal $(K_{n,\infty} \times K_{n-1, \infty})$-subspace of $\tau_n \widehat{\otimes} \tau_{n-1}$.
	Together with the intertwining map $E_{\mu}|_{\mathrm{GL}(n-1)} \otimes E_{\lambda} \rightarrow \mathbb{C}$, this induces a $\pi_0(K_{n-1,\infty}) (\subset \pi_0(K_{n,\infty})\times \pi_0(K_{n-1,\infty}))$-equivariant pairing
	\begin{equation} \label{eqn:pairing-nonzero-infinity}
		\mathrm{H}^{b_n}(\mathfrak{g}_{n}, K_{n,\infty}^0; \Pi_{\infty} \otimes E_{\mu}) \otimes \mathrm{H}^{b_{n-1}}(\mathfrak{g}_{n-1}, K^0_{n-1, \infty}; \Sigma_{\infty} \otimes E_{\lambda}) \rightarrow \mathbb{C}_{\varepsilon},
	\end{equation}
	where $\mathbb{C}_{\varepsilon}$ denotes the nontrivial character of $\pi_0(K_{n-1, \infty})$. \cite[Theorem A]{Sun-nonvanishing} implies that the above pairing is nonzero.
	Let $[\Pi_{\infty}]$ be a nonzero class in $\mathrm{H}^{b_n}(\mathfrak{g}_{n}, K_{n,\infty}^0; \Pi_{\infty} \otimes E_{\mu})$ and $[\Sigma_{\infty}]^{\epsilon}$ denote a nonzero class in $\mathrm{H}^{b_{n-1}}(\mathfrak{g}_{n-1}, K^0_{n-1, \infty}; \Sigma_{\infty} \otimes E_{\lambda})$ on which $\pi_0(K_{n-1, \infty})$ acts via $\epsilon(\Pi_{\infty})\mathrm{sgn}$.
	The pairing in \eqref{eqn:pairing-nonzero-infinity} is nonzero on the subspace $\mathbb{C}[\Pi_{\infty}] \otimes \mathbb{C}[\Sigma_{\infty}]^{\epsilon} \subset \mathrm{H}^{b_n}(\mathfrak{g}_n, K_{n,\infty}^0; \Pi_{\infty} \otimes E_{\mu}) \otimes \mathrm{H}^{b_{n-1}}(\mathfrak{g}_{n-1}, K_{n-1, \infty}^0; \Sigma_{\infty} \otimes E_{\lambda})$.\\
	
	\noindent

	\subsection{Representation theory: Global} \label{subsec:rep-theory-global}
	The subrepresentation $\Sigma_f \otimes \tau_{n-1} \subset \Sigma$ can be realized in $\mathcal{A}(G_{n-1}(\mathbb{Q})\backslash G_{n-1}(\mathbb{A}))$. This is achieved by an Eisenstein series construction.
	Once we verify that the Eisenstein series is holomorphic at the point of interest, we show that $\Sigma_f$ can in fact be realized as a Hecke submodule of $\mathrm{H}^{b_{n-1}}(S_{n-1}, E_{\lambda})$.
	We end this section by establishing Galois equivariance of the embedding constructed via Eisenstein series.

	\subsubsection{Rational structures on automorphic representations} \label{subsubsec:rational-structure-global}
	Let $(\pi_f, V_{\pi_f})$ be a representation of $\mathrm{G}(\mathbb{A}_f)$. For any $\sigma \in \mathrm{Aut}(\mathbb{C})$, define ${}^{\sigma}\pi_f \colonequals V_{\pi_f} \otimes_{\mathbb{C}, \sigma^{-1}}\mathbb{C}$, a representation of $\mathrm{G}(\mathbb{A}_f)$. The field of rationality of $\pi_f$ is defined to be
	\[
	\mathbb{Q}(\pi_f) \colonequals \mathrm{Fix}\left( \lbrace \sigma \in \mathrm{Aut}(\mathbb{C}) \;|\; {}^{\sigma}\pi_f \cong \pi_f \rbrace \right).
	\]
	A representation $\pi_f$ of $\mathrm{G}(\mathbb{A}_f)$ is said to be defined over a field $\mathbb{Q}(\pi_f)$ if there exists a vector space $V_{\pi_f, 0}$ defined over $\mathbb{Q}(\pi_f)$ together with an action of $\mathrm{G}(\mathbb{A}_f)$
	on it such that $\pi_f \xrightarrow{\sim} V_{\pi_f, 0} \otimes_{\mathbb{Q}(\pi_f)} \mathbb{C}$ 
	as $\mathrm{G}(\mathbb{A}_f)$-modules.
	Note that the field of definition and field of rationality coincide for $\mathrm{G} = \mathrm{GL}(n)$.\\
	
	\noindent
	Clozel \cite[Theorem 3.13]{Clozel-motifs} (See also \cite{Januszewski-18} for general reductive groups) proves that any cuspidal automorphic cohomological regular algebraic representation of general linear group is defined over a number field.
	It can be shown using multiplicity one result due to Shalika, that there is the following $G_{n}(\mathbb{A}_f) \times \pi_0(G_{n}(\mathbb{R})) \times \mathrm{Aut}(\mathbb{C}/\overline{\mathbb{Q}})$-equivariant map
	\begin{equation} \label{eqn:comp-aut-Betti-cusp}
		\begin{tikzcd}[column sep = .6in]
			\Pi_f \otimes [\Pi_{\infty}] \arrow[rr, hook, "p^{\mathrm{aut-B}}(\Pi_f)^{-1}\iota_{\Pi}"] && \mathrm{H}^{b_n}_{\text{cusp}}(S_n, E_{\mu})
		\end{tikzcd}
	\end{equation}
	induced by the inclusion $\Pi \lhook\joinrel\xrightarrow{\iota_{\Pi}} \mathcal{A}_{\mathrm{cusp}}(G_n(\mathbb{Q}) \backslash G_n(\mathbb{A}))$.
	We may call this transcendental quantity $p^{\mathrm{aut-B}}(\Pi_f)$ to be automorphic-Betti period as it arises from comparing the automorphic rational structure and the Betti rational structure.\\
	
	\noindent
	It follows from above that $\pi_f$\footnote{appearing in the definition of the automorphic representation $\Sigma$.} is defined over a number field. Using this and the fact that $\chi_i|\cdot|^{\mathbb{Z}/2}$ is defined over $\overline{\mathbb{Q}}$, we obtain that the $G_{n-1}(\mathbb{A}_f)$-representation $\Sigma_f$ is defined over $\overline{\mathbb{Q}}$. More precisely for any $\sigma \in \mathrm{Aut}(\mathbb{C})$, define
	\begin{equation} \label{eqn:rat-str-Sigma}
		{}^{\sigma}\Sigma_f = {}^{un}\mathrm{Ind}^{G_{n-1}(\mathbb{A}_f)}_{P_{n-1}(\mathbb{A}_f)}\left( {}^{\sigma}(\pi_f) \otimes \left( {}^{\sigma}(\chi_{n-2}|\cdot|^{\frac{2\mu_{\lfloor n/2 \rfloor}+1}{2}}) \otimes {}^{\sigma}(\chi_{n-1}|\cdot|^{-\frac{2\mu_{\lfloor n/2 \rfloor}+1}{2}}) \right)_f \otimes \rho_{P_{n-1}}\right).
	\end{equation}
	Note the induction used above is un-normalized and the superscript ${}^{un}$ is added to distinguish this from the unitary induction. With this action of $\mathrm{Aut}(\mathbb{C})$ on $\Sigma_f$, it is easy to see that $\mathbb{Q}(\Sigma_f)$ is a number field. Thus $\Sigma_f$ is defined over a number field and hence over $\overline{\mathbb{Q}}$.\\
	
	\noindent
	Since $\Sigma$ is an induced representation, $\Sigma_f$ may not admit an embedding in the cohomology of locally symmetric space. Although we will finally show that this can be achieved, initially we show that $\Sigma_f$ admits an embedding in the boundary cohomology of locally symmetric space.
	That is, we prove in the paragraphs below that there exists a $G_{n-1}(\mathbb{A}_f) \times \pi_0(G_{n-1}) \times \mathrm{Aut}(\mathbb{C}/\overline{\mathbb{Q}})$-equivariant embedding
	\[
	\begin{tikzcd}[column sep = 1.2in]
		\Sigma_f \otimes [\Sigma_{\infty}]^{\epsilon} \arrow[r, hook, "p^{\mathrm{aut-B}}(\pi\text{,} \Sigma_{\infty})^{-1}\iota_{\Sigma}"] & \mathrm{H}^{b_{n-1}}(\partial_{P_{n-1}} S_{n-1}, E_{\lambda}),
	\end{tikzcd}
	\]
	where $[\Sigma_{\infty}]^{\epsilon}$ is defined as in \S \ref{subsubsec:critical-point} and $p^{\mathrm{aut-B}}(\pi, \Sigma_{\infty}) \in \mathbb{C}^{\times}$ chosen as in the paragraph below. Recall that there is the following decomposition of boundary cohomology
	\[
	\mathrm{H}^{p}_{\mathrm{cusp}}(\partial_{P_{n-1}} S_{n-1}, E_{\lambda}) = \underset{\substack{w \in W^{P_{n-1}}}}{\bigoplus} {}^{un}\mathrm{Ind}^{G_{n-1}(\mathbb{A}_f)}_{P_{n-1}(\mathbb{A}_f)}\mathrm{H}^{p - l(w)}_{cusp}\big(S_{M_{P_{n-1}}}, \mathrm{H}^{l(w)}(\Lie(R_u(P_{n-1})), E_{\lambda})(w\cdot\lambda)\big)^{\overline{M}_{P_{n-1}}}
	\]
	where $\overline{M}_{P_{n-1}} = \ker( \pi_0(M_{P_{n-1}}) \rightarrow \pi_0(G_{n-1}))$ and $\mathrm{H}^{l(w)}(\Lie(R_u(P_{n-1})), E_{\lambda})(w\cdot\lambda)$ denotes the irreducible $M_P$-subrepresentation with extremal weight $w\cdot\lambda$. It remains to produce a $M_{P_{n-1}}(\mathbb{A}_f) \times \pi_0(M_{P_{n-1}}) \times \mathrm{Aut}(\mathbb{C}/\overline{\mathbb{Q}})$-equivariant map
	\begin{equation} \label{eqn:Galois-equivariant-Levi}
		(\pi \otimes \chi \otimes \rho_{P_{n-1}})_f \otimes [(\pi \otimes \chi \otimes \rho_{P_{n-1}})_{\infty}] \lhook\joinrel\xrightarrow{\; p^{\mathrm{aut-B}}(\pi \text{,} \Sigma_{\infty})^{-1}\iota_{\pi \otimes \chi} \;} \mathrm{H}^{b_{n-1} - l(w)}_{cusp}(S_{M_{P_{n-1}}}, \mathrm{H}^{l(w)}(\Lie(R_u(P_{n-1})), E_{\lambda})(w\cdot\lambda))^{\overline{M}_P}
	\end{equation}
	where the class $[(\pi \otimes \chi \otimes \rho_{P_{n-1}})_{\infty}]$ is given by $\epsilon(\Pi_{\infty})\mathrm{sgn} \otimes \dots \otimes \epsilon(\Pi_{\infty})\mathrm{sgn}$ in the Kunneth isomorphism described in (\ref{eqn:Kunneth-iso-Levi}) and $w \in W^{P_{n-1}}$ is of length $\tfrac{(n^2-4n+3)}{4} + 1$.
	It suffices to produce a $M_{P_{n-1}}(\mathbb{A}_f)$-equivariant embedding $\iota_{\pi\otimes \chi}$, it follows from here as before that there exists a quantity $p^{\mathrm{aut-B}}(\pi, \Sigma_{\infty})$ such that the above map is $\mathrm{Aut}(\mathbb{C}/\overline{\mathbb{Q}})$-equivariant.
	This is independent of the character $\chi$ since only the $\overline{\mathbb{Q}}$-structure is in question and $\chi_{n-2}$ and $\chi_{n-1}$ take values in $\overline{\mathbb{Q}}$.
	Define $w = w_{e_{n-2}-e_{n-1}}w_{\lbrace P_{n-1} \rbrace}$, where
	\[
	w_{\lbrace P_{n-1} \rbrace} = \left( \substack{ \begin{aligned} &1\; &2\; &\dots \; &(n-1)/2 \; &\ (n+1)/2 \; &\dots \; &n-3 &\; n-2 \; &\ n-1 \\ &1\; &3\; &\dots \; &n-2 \; &\ \ \ n-1 \; &\dots \; &\ \ \ 6 \; &4 \; & \ \ \ 2 \end{aligned}} \right),
	\]
	is defined as in \cite{Mahnkopf-05}. The length can be computed to be $l(w) = (n^2-4n+3)/4 + 1$. Recall the twisted action of the Weyl group on the characters $w\cdot \lambda = w(\lambda + \rho_{n-1}) - \rho_{n-1}$. A simple computation yields
	\begin{align*}
		w\cdot \lambda = (\mu_1, -\mu_1-n+3, \dots, &\mu_i+i-1, -\mu_i-n+3i, \dots, \mu_{\lfloor n/2 -1 \rfloor}+ \lfloor n/2 -1 \rfloor -1,\\
		&-\mu_{\lfloor n/2 -1 \rfloor}-n+3\lfloor n/2 -1 \rfloor, -\mu_{\lfloor n/2 \rfloor} + (n-5)/2, \mu_{\lfloor n/2 \rfloor} + (n-1)/2).
	\end{align*}
	Let $E_{w\cdot \lambda}$ denotes the representation with extremal weight $w\cdot \lambda$. Consider the representation $(\pi_{\infty}\otimes \chi_{\infty} \otimes \rho_{P_{(2, \dots, 2, 1, 1)}}) \otimes E_{w\cdot\lambda}$ of $\mathrm{GL}(2, \mathbb{R})\times \dots \times \mathrm{GL}(2, \mathbb{R}) \times \mathbb{G}_{m}(\mathbb{R}) \times \mathbb{G}_{m}(\mathbb{R})$.
	Using that the restriction of this representation to $\mathbb{G}_{m} \times \mathbb{G}_{m}$ is $\epsilon(\Pi_{\infty})\mathrm{sgn} \otimes \epsilon(\Pi_{\infty})\mathrm{sgn}$, we get the following isomorphism as $\pi_0(M_{P_{n-1}})$-modules
	\begin{equation} \label{eqn:Kunneth-iso-Levi}
		\begin{split}
			&\mathrm{H}^{(n-3)/2}(\mathfrak{m}_P, (K^{M_P}_{n-1, \infty})^0, (\pi_{\infty}\otimes \chi_{\infty} \otimes \rho_{P_{(2, \dots, 2, 1, 1)}}) \otimes E_{w\cdot\lambda})\\
			&=  \mathrm{H}^{(n-3)/2}\left(\mathfrak{g}_2 \times \dots \times \mathfrak{g}_2, SO(2) \times \dots \times SO(2), (\pi_{\infty} \otimes \rho_{P_{(2, \dots, 2, 1,1)}} \otimes E_{w\cdot \lambda})\right)\\
			&\hspace{2in} \otimes \mathrm{H}^0(\mathbb{G}_{m} \times \mathbb{G}_{m}, \epsilon(\Pi_{\infty})\mathrm{sgn}\otimes \epsilon(\Pi_{\infty})\mathrm{sgn}).
		\end{split}
	\end{equation}
	Now taking the invariants under $\ker (\pi_0(M_{P_{n-1}}) \rightarrow \pi_0(G_{n-1}))$, we get the class represented by $\epsilon(\Pi_{\infty})\mathrm{sgn} \otimes \dots \otimes \epsilon(\Pi_{\infty})\mathrm{sgn}$ survives. Thus we get that the desired map from \eqref{eqn:Galois-equivariant-Levi} exists. Applying un-normalized induction we get the desired embedding
	\[
	{}^{un}\mathrm{Ind}\left( \pi \otimes \chi \otimes \rho_{P_{n-1}}\right)_f \otimes \epsilon(\Pi_{\infty})\mathrm{sgn} \lhook\joinrel\xrightarrow{\iota_{\Sigma}} \mathrm{H}^{b_{n-1}}(\partial_{P_{n-1}}S_{n-1}, E_{\lambda}).
	\]
	Since $p^{\mathrm{aut-B}}(\pi, \Sigma_{\infty})^{-1}\iota_{\pi \otimes \chi}$ is $M_{P_{n-1}}(\mathbb{A}_f) \times \pi_0(G_{n-1}(\mathbb{R})) \times \mathrm{Aut}(\mathbb{C}/\overline{\mathbb{Q}})$-equivariant map, the map $\iota_{\Sigma}$ has the desired properties.

	\begin{remark}
		Since we are only interested in algebraicity results for $L$-values, we will only be concerned with $\overline{\mathbb{Q}}$-structures and ignore any subtleties arising from the fact that the objects of interest might be defined over some particular number field.
	\end{remark}

	\subsubsection{Intertwining operators : zeros and poles}
	To verify that Eisenstein series map defined in \eqref{eqn:defining-Eis-map} is well defined at $s=2\mu_{\lfloor n/2 \rfloor}+1$, it suffices to check that the intertwining operators which appear in the computation of the constant term of the Eisenstein series are well defined at the point of evaluation.
	Recall that the intertwining operators which appear in this computation are denoted by $M(w, s)$, where $w$ is an element of the Weyl group of $G_{n-1}$ of minimal length modulo the Weyl group of $M_{P_{n-1}}$ such that $wM_{P_{n-1}}w^{-1}$ is again a standard Levi of $G_{n-1}$. As is the convention, this subset is denoted by $W(M_{P_{n-1}})$.
	Let $r(w, s)$ denote the normalization factors for the intertwining operators $M(w, s)$ as dictated by the work of M{\oe}glin-Waldspurger \cite{Moeglin-Waldspurger-residual}.
	It is known that the poles and zeros of the global intertwining operator is governed by the normalizing factors which are a ratio of certain global $L$-functions.
	\begin{lemma} \label{lemma:intertwine-iso}
		The global intertwining operators $M\big(w, s\big)$ are invertible for $s=2\mu_{\lfloor n/2 \rfloor}+1$ and $w \in W(M_{P_{n-1}})$ such that $w(e_{n-2}-e_{n-1})$ is again positive root.
	\end{lemma}
	\begin{proof}
		Let $T_{M_{P_{n-1}}} \subset T_{n-1}$ be the maximal torus contained in the center of $M_{P_{n-1}}$ and let $\Phi(T_{M_{P_{n-1}}}, G_{n-1})$ denote the root system. For each simple element $\alpha \in \Phi(T_{M_{P_{n-1}}}, G_{n-1})$, let $w_{\alpha} \in W(M_{P_{n-1}})$ denote the corresponding element.
		It is known that elements $w_{\alpha}$ generate $W(M_{P_{n-1}})$ where $\alpha \in \Phi(T_{M_{P_{n-1}}}, G_{n-1})$. Let $w_{\alpha}$ be an element of $W(M_{P_{n-1}})$ where $\alpha$ is not the root $e_{n-2}-e_{n-1}$.
		Since $\alpha$ is simple, the corresponding intertwining operators and their normalizing factors are of the form listed below
		\[
		\begin{tikzcd}[column sep = .5 ex]
			\mathrm{Ind}^{G_4}_{P_{(2,2)}}(\pi_i \otimes \pi_{i+1}) \rightarrow \mathrm{Ind}^{G_4}_{P_{(2,2)}}(\pi_{i+1} \otimes \pi_{i}) & r = \frac{L(0, \pi_i \times \pi^{\vee}_{i+1})}{L(1, \pi_i \times \pi^{\vee}_{i+1})}.\\
			\mathrm{Ind}^{G_3}_{P_{(2,1)}}(\chi_{n-2}|\cdot|^{\mu_{\lfloor n/2 \rfloor}+\frac{1}{2}} \times \pi_i) \rightarrow \mathrm{Ind}^{G_3}_{P_{(2,1)}}(\pi_i \times \chi_{n-2}|\cdot|^{\mu_{\lfloor n/2 \rfloor}+\frac{1}{2}}) & r = \frac{L(\mu_{\lfloor n/2 \rfloor}+\frac{1}{2}, \chi_{n-2} \times \pi^{\vee}_i)}{L(\mu_{\lfloor n/2 \rfloor}+\frac{3}{2}, \chi_{n-2} \times \pi^{\vee}_i)}.\\
			\mathrm{Ind}^{G_3}_{P_{(2,1)}}(\pi_i \times \chi_{n-1}|\cdot|^{-\mu_{\lfloor n/2 \rfloor}-\frac{1}{2}}) \rightarrow \mathrm{Ind}^{G_3}_{P_{(2,1)}}(\chi_{n-1}|\cdot|^{-\mu_{\lfloor n/2 \rfloor}-\frac{1}{2}} \times \pi_i) & r = \frac{L(\mu_{\lfloor n/2 \rfloor}+\frac{1}{2}, \pi_i \times \chi^{-1}_{n-1})}{L(\mu_{\lfloor n/2 \rfloor}+\frac{3}{2}, \pi_i \times \chi^{-1}_{n-1})}.
		\end{tikzcd}
		\]
		Since $\pi_i$'s are distinct and they satisfy $\mathbf{Assumption(\mu, \chi_{n-2}, \chi_{n-1})}$, all the normalizing factors are nonzero. Hence the intertwining operators $M(w_{\alpha}, s)$ are all nonzero.
		Note that the representations $\mathrm{Ind}^{\mathrm{GL}(4)}_{P_{(2,2)}}(\pi_i \times \pi_{i+1})$ for unitary cuspidal representations $\pi_i \neq \pi_{i+1}$, and $\mathrm{Ind}^{\mathrm{GL}(3)}_{P_{(2,1)}}(\pi_i \times \chi|\cdot|^s)$ for $\Re{(s)} > 0$ are irreducible. Hence the intertwining operators are isomorphisms.
	\end{proof}

	\section{Eisenstein series construction and its rationality properties} \label{sec:global-map-Eis-non-zero}
	
	Consider the map
	\begin{equation} \label{eqn:defining-Eis-map}
		\begin{tikzcd}[column sep = 1ex]
			\mathrm{Ind}^{G_{n-1}(\mathbb{A})}_{P_{n-1}(\mathbb{A})}\left( \pi\otimes \chi_{n-2}|\cdot|^{s/2} \otimes \chi_{n-1}|\cdot|^{-s/2}\right) \arrow[r, "\mathrm{Eis(s, \cdot, \cdot)}"] & \mathcal{A}\left( G_{n-1}(\mathbb{Q})\backslash G_{n-1}(\mathbb{A})\right);\\
			\phi(s, \cdot)  \longmapsto & \hspace{-1in}\mathrm{Eis}(s,\phi, g) \colonequals \underset{\gamma \in P_{n-1}(\mathbb{Q})\backslash G_{n-1}(\mathbb{Q})}{\sum} e^{\langle H_{P_{n-1}}(\gamma g), \chi(s)+\rho_{P_{n-1}} \rangle} \phi(\gamma g),
		\end{tikzcd}
	\end{equation}
	where $\chi(s)$ is the character of the torus given by
	\[
	diag(*, a_{n-2},a_{n-1}) \mapsto \chi_{n-2}(a_{n-2})\chi_{n-1}(a_{n-1})|a_{n-2}|^{s/2}|a_{n-1}|^{-s/2}.
	\]
	Recall that, the behaviour of Eisenstein series is dictated by the constant term of the Eisenstein series for standard parabolic subgroups $P_{n-1}'$ that are associate to $P_{n-1}$.

	\subsection{Holomorphy and injectivity of $\mathrm{Eis}$ map}
	Consider the restriction of the map $\mathrm{Eis}$ to the subrepresentation $\Sigma_f \otimes \tau_{n-1} \subset \Sigma$. We continue to denote this map by $\mathrm{Eis}$.

	\begin{prop}
		The map $\mathrm{Eis}$ is defined at $s = 2\mu_{{\lfloor n/2 \rfloor}}+1$.
	\end{prop}
	
	\begin{proof}
		Recall that the intertwining operators $M(w, s)$ for $w \in W(M_{P_{n-1}})$ such that $w_{e_{n-2}-e_{n-1}}$ does not appear in the decomposition of $w$ are isomorphisms and in particular are all well-defined.
		Note that $M(w_{\alpha}, s)$ where $\alpha = e_{n-2}-e_{n-1}$ vanishes when restricted to $\Sigma_f \otimes \tau_{n-1}$. Hence $M(w, s)$ is well-defined at $s = 2\mu_{{\lfloor n/2 \rfloor}}+1$ for all $w$ which appear in the computation of the constant term of the Eisenstein series. Thus we can conclude that the map $\mathrm{Eis}$ is well-defined as well.
	\end{proof}
	
	\noindent
	Thus we may consider the map induced on cohomology
	\begin{align} \label{eqn:Eis}
		&\mathrm{Eis}^{\mathrm{coh}} : \Sigma_f \otimes \mathrm{H}^{b_{n-1}}\left(\mathfrak{g}_{n-1}, K^0_{n-1, \infty}Z_{n-1}(\mathbb{R})^0; \tau_{n-1} \otimes E_{\lambda}\right) \nonumber \\
		&\hspace{1.3in} \longrightarrow \mathrm{H}^{b_{n-1}}\left(\mathfrak{g}_{n-1}, K^0_{n-1,\infty}Z_{n-1}(\mathbb{R})^0; \mathcal{A}(G_{n-1}(\mathbb{Q})\backslash G_{n-1}(\mathbb{A})\right).
	\end{align}

	\begin{prop}
		The map $\mathrm{Eis}^{\mathrm{coh}}$ is injective.
	\end{prop}
	
	\begin{proof}
		The idea is to proceed as in the work of Harder and Mahnkopf \cite{Harder-87, Mahnkopf-98}. We know from \cite[Satz 1.10]{Schwermer-Kohomologie} that the class of the restriction of a closed differential form $\omega$ to the boundary $\partial_{P_{n-1}}S_{n-1}$ is the same as the class of the constant term of $\omega$ along $P_{n-1}$.
		Recall that for $\phi_{\infty}$ which is any element of $\tau_{n-1}$, then $M((3,4), 2\mu_{\lfloor n/2 \rfloor}+1)\phi_{\infty} = 0$. This implies that for any $\phi_{f}\otimes \phi_{\infty} \in \Sigma_f \otimes \tau_{n-1}$,
		\[
		\mathrm{Eis}^{U_{P_{n-1}}}(\phi_f \otimes \phi_{\infty}) = \phi_f \otimes \phi_{\infty} + M((3, 4), 2\mu_{\lfloor n/2 \rfloor}+1)(\phi_f\otimes \phi_{\infty}) = \phi_f \otimes \phi_{\infty}.
		\]
		Thus the map $\mathrm{Eis}^{\mathrm{coh}}$ composed with the constant term map, which is the same as considering constant term of the elements in the image of $\mathrm{Eis^{coh}}$ along the parabolic subgroup $P_{n-1}$ is identity. Hence the construction above gives us an injective map.\\
	\end{proof}

	\subsection{$\mathrm{Aut}(\mathbb{C}/\overline{\mathbb{Q}})$-equivariance of $\mathrm{Eis}^{\mathrm{coh}}$ map}
	In this subsection we consider the question of comparing the rational structures on the domain and target of the map in equation \eqref{eqn:Eis}. The rational structure on the domain is recalled in \S \ref{subsubsec:rational-structure-global} and the rational structure on the target is the Betti rational structure.\\
	
	\begin{prop} \label{prop:Eis-rationality}
		The map
		\begin{equation} \label{eqn:}
			\Sigma_f \xrightarrow{\quad p^{\mathrm{aut-B}}(\pi, \Sigma_{\infty})^{-1}\mathrm{Eis^{coh}}\quad} \mathrm{H}^{b_{n-1}}(S_{n-1}, E_{\lambda})
		\end{equation}
		is $\textup{Aut}(\mathbb{C}/\overline{\mathbb{Q}})$-equivariant with respect to the $\overline{\mathbb{Q}}$-structure on the representation $\Sigma_f$ described in \S \ref{subsubsec:rational-structure-global} and the natural $\overline{\mathbb{Q}}$-structure on Betti cohomology.
	\end{prop}

	\begin{proof}
		The claim in the proposition is equivalent to proving the commutativity of the diagram below for any $\sigma \in \mathrm{Aut}(\mathbb{C}/\overline{\mathbb{Q}})$.
		\[
		\begin{tikzcd}[column sep = 1.2in]
			\Sigma_f \arrow[d, "\sigma"'] \arrow[r, "p^{\mathrm{aut-B}}(\pi\text{,} \Sigma_{\infty})^{-1}\mathrm{Eis^{coh}}"] & \mathrm{H}^{b_{n-1}}(S_{n-1}, E_{\lambda}) \arrow[d, "\sigma"]\\
			\Sigma_f \arrow[r, "p^{\mathrm{aut-B}}(\pi\text{,} \Sigma_{\infty})^{-1}\mathrm{Eis^{coh}}"'] & \mathrm{H}^{b_{n-1}}(S_{n-1}, E_{\lambda})
		\end{tikzcd}
		\]
		To prove this we can consider the difference of the two composite arrows and show that it is identically zero. Note that \cite[Theorem 4.3]{FS} implies that for any $\sigma \in \mathrm{Aut}(\mathbb{C}/\overline{\mathbb{Q}})$,
		\[
		\mathrm{Eis^{coh}}(\phi^{\sigma}_f), (\mathrm{Eis^{coh}}(\phi_f))^{\sigma} \in \mathrm{H}^{b_{n-1}}\left(\mathfrak{g}_{n-1}, K_{n-1,\infty}^0; \mathcal{A}_{(\lbrace P_{n-1}\rbrace, \pi, \chi_{n-2}, \chi_{n-1})}(G_{n-1}(\mathbb{Q})\backslash G_{n-1}(\mathbb{A})) \otimes E_{\lambda}\right).
		\]
		By hypothesis the vector at infinity is annihilated by the intertwining operator, hence it follows as in \cite[\S 4.2, Page 82]{Harder-87} that the constant term along the parabolic $P_{n-1}$ of the differential form $\mathrm{Eis^{coh}}(\phi_f)$ is the differential form that we started with.
		That is
		\begin{equation} \label{eqn:rest-zero}
			\begin{split}
				\mathrm{res}_{P_{n-1}}([p^{\mathrm{aut-B}}(\pi, \Sigma_{\infty})^{-1}\mathrm{Eis^{coh}}(\phi^{\sigma}_f)] - [p^{\mathrm{aut-B}}(&\pi, \Sigma_{\infty})^{-1}\mathrm{Eis^{coh}}(\phi_f)]^{\sigma})\\
				&= [p^{\mathrm{aut-B}}(\pi, \Sigma_{\infty})^{-1}\iota_{\Sigma}(\phi^{\sigma}_f)] - [p^{\mathrm{aut-B}}(\pi, \Sigma_{\infty})^{-1}\iota_{\Sigma}(\phi_f)]^{\sigma}\\
				&= 0.
			\end{split}
		\end{equation}
		\noindent
		for any $\sigma \in \mathrm{Aut}(\mathbb{C}/\overline{\mathbb{Q}})$. The last equality follows from the $\mathrm{Aut}(\mathbb{C}/\overline{\mathbb{Q}})$-equivariance properties of the map $p^{\mathrm{aut-B}}(\pi, \Sigma_{\infty})^{-1}\iota_{\Sigma}$.\\
		
		\noindent
		To analyze further we need to understand the summand $\mathcal{A}_{\lbrace P_{n-1} \rbrace, \pi, \chi_{n-2}, \chi_{n-1}}$(See footnote \footnote{Here we have suppressed the notation $\mathcal{J}_{E_{\lambda}}$ for the annihilator ideal of the finite dimensional representation $E_{\lambda}$ for the action of the the center of universal enveloping algebra of $\mathfrak{g}_{n-1}$.}) of automorphic forms.
		Since the cuspidal datum $(\lbrace P_{n-1} \rbrace, \pi, \chi_{n-2}, \chi_{n-1})$ do not have any nontrivial automorphism, the Franke filtration on the summand $\mathcal{A}_{\lbrace P_{n-1} \rbrace, \pi, \chi_{n-2},  \chi_{n-1}}$ is degenerate.
		Using \cite[Theorem 14]{Franke} and its refinement \cite[Theorem 4]{Grobner-residual} we have the following (See also \cite{Waldspurger-d'apres-Franke}),
		\begin{equation} \label{eqn:description-aut-forms}
			\mathrm{Ind}^{G_{n-1}}_{P_{n-1}}\Big(\pi \otimes \chi \otimes S(\mathfrak{a}^G_{P_{n-1}})\Big) \overset{\sim}{\longrightarrow} \mathcal{A}_{\lbrace P_{n-1} \rbrace, \pi, \chi_{n-2}, \chi_{n-1}}
		\end{equation}
		
		\noindent
		Let $\Omega^*(\mathfrak{g}_{n-1}, K^0_{n-1,\infty}, \mathcal{A}_{\lbrace P_{n-1} \rbrace, \pi, \chi_{n-2}, \chi_{n-1}} \otimes E_{\lambda})$ denote the complex of differential forms with functions in the subspace $\mathcal{A}_{\lbrace P_{n-1} \rbrace, \pi, \chi_{n-2}, \chi_{n-1}}$.
		From here we argue as in \cite[\S 4.6]{Moeglin-00} to prove that a nonzero closed differential form up to coboundaries in the chain complex $\Omega^*(\mathfrak{g}_{n-1}, K_{n-1,\infty}, \mathcal{A}_{\lbrace P_{n-1} \rbrace, \pi, \chi_{n-2}, \chi_{n-1}}$ $\otimes E_{\lambda})$ restricts to a nonzero cohomology class on the boundary $\partial_{P_{n-1}}S_{n-1}$.
		Assuming that the above holds, the proof of the proposition is complete in light of \eqref{eqn:rest-zero}.\\
		
		\noindent
		Suppose $\omega$ is a closed differential form in $\Omega^*(\mathfrak{g}_{n-1}, K^0_{n-1,\infty}, \mathcal{A}_{\lbrace P_{n-1} \rbrace, \pi, \chi_{n-2}, \chi_{n-1}} \otimes E_{\lambda})$. Using the structure of $\mathcal{A}_{\lbrace P_{n-1} \rbrace, \pi, \chi_{n-2}, \chi_{n-1}}$ from \eqref{eqn:description-aut-forms} and \cite[pp. 256-257]{Franke}, we get an isomorphism induced by the map $\mathrm{Eis}$
		\[
		\Sigma_f \otimes \mathrm{H}^{b_{n-1}}(\mathfrak{g}_{n-1}, K^0_{n-1, \infty}, \Sigma_{\infty} \otimes E_{\lambda}) \xrightarrow{\sim} \mathrm{H}^{b_{n-1}}(\mathfrak{g}_{n-1}, K^0_{n-1, \infty}, \mathcal{A}_{\lbrace P_{n-1} \rbrace, \pi \otimes \chi})
		\]
		This implies that up to coboundaries there exists $\Omega \in \Omega^{b_{n-1}}(\mathfrak{g}_{n-1}, K^0_{{n-1},\infty}; \mathrm{Ind}^{G_{n-1}}_{P_{n-1}}(\pi \otimes \chi))$ such that
		\[
		\omega = \mathrm{Eis}(\Omega).
		\]
		Let $\omega_{P_{n-1}}$ be the constant term along the parabolic subgroup $P_{n-1}$ and we denote by $\omega_{P_{n-1}, (\pi \otimes \chi)}$ the cuspidal projection of the constant term along the representation $\pi \otimes \chi$. Note that we have $\omega_{P_{n-1}, (\pi \otimes \chi)} = \Omega$.
		Define the trivial subcomplex of $\Omega^*(\mathfrak{g}_{n-1}, K^0_{n-1,\infty}, \mathcal{A}_{\lbrace P_{n-1} \rbrace, \pi, \chi_{n-2}, \chi_{n-1}} \otimes E_{\lambda})$ to be the subcomplex on which the center of $\mathfrak{m}_{P_{n-1}}$ acts trivially under the following isomorphism
		\[
		\Omega^l(\mathfrak{g}_{n-1}, K^0_{{n-1},\infty}; \mathrm{Ind}^{G_{n-1}}_{P_{n-1}}(\pi \otimes \chi)) \simeq \bigoplus_{i+j =l} \Omega^i\left( \mathfrak{m}_{P_{n-1}}, (K^{P_{n-1}}_{n-1, \infty})^0; \pi \otimes \chi\otimes \rho_{P_{n-1}} \otimes \hom(\wedge^j\mathfrak{u}_{P_{n-1}}, E_{\lambda})\right).
		\]
		Recall that $\Omega$ can be chosen in the trivial subcomplex up to coboundaries. This assures us that if $\Omega$ is not zero then its class is not zero in $\mathrm{H}^{b_{n-1}}\Omega^*(\mathfrak{g}_{n-1}, K^0_{n-1, \infty};$ $\mathrm{Ind}^{G_{n-1}}_{P_{n-1}} (\pi \otimes \chi)\otimes E_{\lambda})$.
		Thus if $\Omega$ is nonzero, then $\omega_{P_{n-1}} \neq 0$. That is the restriction of the cohomology class $[\omega]$ to the face of a Borel-Serre compactification corresponding to $P_{n-1}$ is also nonzero.
		This completes the proof of the assertion in the paragraph above.
	\end{proof}

	\subsection{Rational structure on Whittaker models and Betti-Whittaker periods} \label{subsec:rat-Whittaker-Betti-periods}
	Let us fix an additive character $\psi : \mathbb{Q}\backslash \mathbb{A} \rightarrow \mathbb{C}$. Recall that cuspidal automorphic representations admit Whittaker models. Fix an isomorphism $\mathcal{F}_\Pi^{\mathrm{W-aut}} : \Pi_f \simeq \mathcal{W}(\Pi_f, \psi)$ given by the $\psi$-th Fourier coefficient. The Galois action on $\mathcal{W}(\Pi_f, \psi)$ is defined as in \cite[\S 4]{Mahnkopf-05}.
	Comparing the rational structure on $\mathcal{W}(\Pi_f, \psi)$ and $\Pi_f$ via the above isomorphism reveals that they differ by a complex number $p^{\mathrm{W-aut}}(\Pi_f)$. Consider the composite map where the first map is the inverse of $\mathcal{F}_\Pi^{\mathrm{W-aut}}$
	\begin{equation} \label{eqn:Whittaker-Betti}
		\mathcal{F}_{\Pi_f} : \mathcal{W}(\Pi_f, \psi) \simeq \Pi_f \xrightarrow{\iota_{\Pi}} \mathrm{H}^{b_n}(S_{n}, E_{\mu}).
	\end{equation}
	With $p(\Pi_f) = p^{\mathrm{W-aut}}(\Pi_f)p^{\mathrm{aut-B}}(\Pi_f)$, the following identity holds : $p(\Pi_f)\sigma\circ \mathcal{F}_{\Pi_f} = \sigma(p(\Pi_f)) \mathcal{F}_{\Pi_f} \circ \sigma$ for any $\sigma \in \mathrm{Aut}(\mathbb{C}/\overline{\mathbb{Q}})$.\\

	\noindent
	Similarly, let $\mathcal{F}^{\mathrm{W-aut}}_{\pi} : \pi_f \rightarrow \mathcal{W}(\pi_f, \psi^{-1})$ and $p^{\mathrm{W-aut}}(\pi_f)$ be defined as above for the cuspidal automorphic representations $\pi$ of the Levi $M_{P_{n-1}}$ of the parabolic subgroup $P_{n-1}$.
	Consider the composite where the first map is inverse of $(\mathcal{F}^{\mathrm{W-aut}}_{\pi} \times \mathrm{id})$
	\begin{equation}
		\mathcal{F}_{\Sigma_f} : \mathrm{Ind}(\mathcal{W}(\pi_f) \times \chi_f) \rightarrow \mathrm{Ind}(\pi_f \times \chi_f) \lhook\joinrel\xrightarrow{\mathrm{Eis^{coh}}} \mathrm{H}^{b_{n-1}}(S_{n-1}, E_{\lambda}).
	\end{equation}
	Let $\tilde{p}(\Sigma_f) = p^{\mathrm{aut-B}}(\pi, \Sigma_{\infty})p^{\mathrm{W-aut}}(\pi_f)$, and the induced representation is equipped with $\overline{\mathbb{Q}}$-structure as in \S \ref{subsubsec:rational-structure-global} then we have the equality $\tilde{p}(\Sigma_f)\sigma\circ \mathcal{F}_{\Sigma_f} = \sigma(\tilde{p}(\Sigma_f))\mathcal{F}_{\Sigma_f}\circ \sigma$ for any $\sigma \in \mathrm{Aut}(\mathbb{C}/\mathbb{Q})$.

	\section{Rankin--Selberg $L$-function for $G_n \times G_{n-1}$}
	\subsection{Analytic theory of Rankin--Selberg $L$-functions for $G_n \times G_{n-1}$} \label{subsec:5.1}
	For $t \in \mathbb{C}$ and $\Re(t)\gg 0$, we consider the induced representation $\Sigma_t = \Sigma \otimes \rho^t_{P_{n-1}}$, where $\Sigma$ is defined in \eqref{eqn:def-Sigma}.
	For $\Re(t)\gg 0$ the Eisenstein series is well-defined due to generalities, let $V_{\Sigma_t} = \mathrm{Eis}(\Sigma_t)$.\\
	
	\noindent
	We will apply the Rankin--Selberg theory of $L$-functions for $G_n \times G_{n-1}$ to the pair $(\Pi, \Sigma_t)$, where  $(\Pi, V_{\Pi})$ is a cuspidal automorphic representation of $G_n(\mathbb{A})$ and $(\Sigma_t, V_{\Sigma_t})$ is the induced representation.
	Let $\phi_{\Pi} \in V_{\Pi}$ and $\phi_{\Sigma_t} \in V_{\Sigma_t}$, then $\phi_{\Pi}$ and $\mathrm{Eis}(\phi_{\Sigma_t})$ are automorphic forms on the group $G_{n}$ and $G_{n-1}$, respectively. Recall that $\iota$ denotes the embedding of $G_{n-1}$ in $G_n$ in the upper left corner. Consider the global period integral
	
	\begin{equation} \label{eqn:period-integral}
		I(s, \phi_\Pi, \mathrm{Eis}(\phi_{\Sigma_t})) \colonequals \underset{G_{n-1}(\mathbb{Q}) \backslash G_{n-1}(\mathbb{A})}{\int} \phi_\Pi(\iota(g)) \mathrm{Eis}(\phi_{\Sigma_t})(g)|\det(g)|^{s- \frac{1}{2}}dg. 
	\end{equation}
	This integral converges for $\Re(s) = \frac{1}{2}$ since a cusp form decays rapidly whereas an Eisenstein series (or its residue) is of moderate growth. For a nontrivial additive character $\psi : \mathbb{Q} \backslash \mathbb{A} \to \mathbb{C}$, suppose that $w_\Pi \in \mathcal{W}(\Pi_f, \psi)$, and $w_{\Sigma_t} \in \mathcal{W}((\Sigma_t)_f, \psi^{-1})$ are global Whittaker vectors corresponding to $\phi_\Pi$ and $\mathrm{Eis}(\phi_{\Sigma_t}),$ respectively.
	Moreover assume that $\phi_\Pi$ and $\phi_{\Sigma_t}$ are chosen so that $w_\Pi = \otimes' w_{\Pi_v}$ and $w_{\Sigma_t} = \otimes' w_{\Sigma_{t,v}},$ are pure tensors. The Euler factorization of the integrand leads to a similar factorization of the above period integral. Finally we have 
	\begin{align*}
		I\left(\frac{1}{2}, \phi_\Pi, \mathrm{Eis}(\phi_{\Sigma_t})\right) &\colonequals \underset{{U_{n-1}(\mathbb{A})\backslash \mathrm{GL}(n-1,\mathbb{A})}}{\int} w_\Pi(\iota(g)) w_{\Sigma_t}(g)dg\\
		&=\prod_v \underset{{U_{n-1}(\mathbb{Q}_v)\backslash \mathrm{GL}(n-1,\mathbb{Q}_v)}}{\int} w_{\Pi_v}(\iota(g_v)) w_{\Sigma_{t,v}}(g_v) dg_v\\
		&\equalscolon \prod_v \Psi_v\left(\frac{1}{2}, w_{\Pi_v}, w_{\Sigma_{t,v}}\right).
	\end{align*}
	See \cite{Cogdell} for more details. In the next section, we pin down the choices of $w_{\Pi}$ and $w_{\Sigma_t}$, and compute the local integrals $\Psi(s, w_{\Pi_v}, w_{\Sigma_{t,v}})$ at unramified and ramified places. Since we are only concerned with $L$-values up to a nonzero algebraic number, we need not compute the local integrals at ramified places but only show that they are a nonzero algebraic number.

	\subsection{Choice of local Whittaker vectors and local Rankin--Selberg integrals} \label{subsec:choice-of-local-Whittaker-vectors-local-integral}
	\subsubsection{Two lemmas}
	It is known from \cite{Casselman-Shalika-II} that $\Sigma_{t,v}$ admits a Whittaker model for all $\Re(t) \gg 0$.
	The choices of local Whittaker vectors at places which are unramified for $\Sigma_t$ is dictated by the work of Jacquet-Piateski-Shapiro-Shalika \cite{Jacquet-Shapiro-Shalika}. Although the choices of Whittaker vectors at places which are ramified for $\Sigma_t$ is not clear, it is possible (see lemma below) to choose vectors which under the local Rankin--Selberg integral is a nonzero algebraic number.
	This suffices for our purpose of studying algebraic aspects of the theory of $L$-functions and is stated in the lemma below.

	\begin{lemma} \label{lemma:local-Whittaker-choices-ramified}
		Let $F$ be a non-archimedean local field and $\psi_F$(See footnote \footnote{This notation is used to differentiate from the notation of the global additive character $\psi$ used previously.})
		an additive character of $F$. Let $\pi_1$ and $\pi_2$ be two representations of $G_n(F)$ and $G_{n-1}(F)$, respectively which admit Whittaker models. Assume that the Whittaker models are furthermore defined over $\overline{\mathbb{Q}}$. There exists $w_1 \in W(\pi_1, \psi_F)$ and $w_{2} \in W(\pi_2, \psi^{-1}_F)$ satisfying the following conditions
		\begin{enumerate}
			\item The functions $w_1$ and $w_2$ are $\overline{\mathbb{Q}}$-valued.
			\item The local Rankin--Selberg integral
			\[
			\Psi(s, w_1, w_2) = \underset{U_{n-1}(F)\backslash \mathrm{GL}_{n-1}(F)}{\int} w_1(g)w_2(g)|\det(g)|^{s-\frac{1}{2}}dg
			\]
			is given by $P_{n-1}(q^{-s})$ where $P_{n-1}(T) \in \overline{\mathbb{Q}}[T, T^{-1}]$.
		\end{enumerate}
	\end{lemma}
	\begin{proof}
		The lemma is a consequence of \cite{Jacquet-Shapiro-Shalika} as shown in \cite[Lemma 3.7]{Grobner-Harris} and \cite{Mahnkopf-05}. Only the case when $\pi_2$ is ramified requires explanation. Let $w_{\pi_2}$ be the unique vector fixed under a mirabolic subgroup $K_{n-1}$ of some level.
		Since $W(\pi_2, \psi^{-1}_F)$ is defined over $\overline{\mathbb{Q}}$ and the space of $K_{n-1}$-invariant vectors in the complex vector space $W(\pi_2, \psi^{-1}_F)$ is one dimensional, hence $w_{\pi_2}$ is defined over $\overline{\mathbb{Q}}$.
		Then using the right $K_{n-1}$ and left $U_{n-1}(F)$ invariance properties of the vector $w_{\pi_2}$, we may assume that $w_{\pi_2}$ does not vanish on some $t_0 \in T_{n-1}(F)$.
		Note that $K_{n-1}$ is a mirabolic subgroup and hence the vector $w_{\pi_2}$ is right $U_{n-1}(\mathcal{O}_F)$-invariant as well. Consequently, the vector $w_{\pi_2}$ is only supported on the subset
		\[
		T^+_{n-1}(F) \colonequals \lbrace t = \diag(t_1, t_2, \dots, t_{n-1}) \suchthat t_it^{-1}_{i+1} \in \mathcal{O}_F \text{ for all } i \rbrace.
		\]
		Let $w_{\pi_2}$ be a Whittaker function as described above and renormalized such that $w_{\pi_2}(t_0) =1$. Choose $w_{\pi_1}$ to be the unique function such that its restriction to $G_{n-1}(F)$ is supported on $U_{n-1}(F)t_0K_{n-1}$ and is given by $w_{\pi_1}(ut_0k) = \psi(u)$.
		The existence of such a function is guaranteed by \cite[Theorem 5]{Gelfand-Kazhdan-72}. Arguing as before the uniqueness of the function $w_{\pi_1}$ shows that the vector $w_{\pi_1}$ is defined over $\overline{\mathbb{Q}}$. Computing the integral explicitly we get that
		\[
		\Psi_v(s, w_{\pi_1}, w_{\pi_2}) = \underset{U_{n-1}(F)\backslash U_{n-1}(F)tK_{n-1}}{\int}| w_{\pi_2}(t)\det(g)|^{s-\frac{1}{2}} dg \in |\det(t)|^{s-\frac{1}{2}}\mathrm{vol}(K_{n-1})\overline{\mathbb{Q}}.
		\]
		The lemma follows.
	\end{proof}
	
	\noindent
	Recall the following commutative diagram from \cite[\S 1.2]{Mahnkopf-05}
	\begin{equation} \label{eqn:comm-diagram-indWhittaker-Whittakerinduced}
		\begin{tikzcd}[column sep = 1in]
			\mathrm{Ind}(\mathcal{W}(\pi)\times \chi) \arrow[d, "\mathcal{F}^{\textup{loc}}"'] & \mathrm{Ind}(\pi\times \chi) \arrow[d, "\textup{Eis}"] \arrow[l, "\mathcal{F}^{\mathrm{W-aut}}_{\pi}"']\\
			\mathcal{W}(\Sigma) & \Sigma \arrow[l, "\mathcal{F}^{\mathrm{W-aut}}_{\Sigma}"],
		\end{tikzcd}
	\end{equation}
	where $\mathcal{F}^{\mathrm{loc}}$ is defined as follows:
	\[
	\mathcal{F}^{\mathrm{loc}}(f^{\mathcal{W}})(g) = \prod_v \mathcal{F}_v(f^{\mathcal{W}}_v)(g) \colonequals \prod_v \underset{(U^{-}_P)^{w_0w_P}}{\int} f_v^{\mathcal{W}}(w_Pw_0ug)\psi_v(u) du.
	\]

	\noindent
	We will need the following lemma in \S \ref{subsubsec:Whittaker-vector-choices-list}.
	
	\begin{lemma} \label{lemma:intertwiner-rational-whittaker-model}
		Let $F$ be a non-archimedean field, $\pi$ be an irreducible admissible generic representation of $G_{n-1}(F)$ defined over $\overline{\mathbb{Q}}$ and $\chi$ a character of $F^*$ valued in $\overline{\mathbb{Q}}$.
		Assume that $\mathrm{Ind}(\pi \times \chi)$ is irreducible. Then the map
		\[
		\mathcal{F} : \mathrm{Ind}\left( \mathcal{W}(\mathcal{\pi}) \times \chi\right) \rightarrow \mathcal{W}\left( \mathrm{Ind}(\pi \times \chi)\right),
		\]
		described in the above equation is defined over $\overline{\mathbb{Q}}$. Here $\mathrm{Ind}$ denotes the normalized induction functor.
	\end{lemma}
	\begin{proof}
		This is simply a restatement of \cite[Prop. 1.4.2, Prop. 1.4.4]{Mahnkopf-05}. Let $w_{\pi}$ be an essential vector in $\pi$ and $t = \diag(t_1, \dots, t_{n-2}) \in T_{n-2}(F)$ such that $w_{\pi}(t) = 1$. Define the Whittaker vector $f^{\mathcal{W}} \in \mathrm{Ind}(\mathcal{W}(\pi) \times \chi)$ as in \cite[(1.6), (1.7)]{Mahnkopf-05}.
		With this definition, \cite[Prop. 1.4.2, Prop. 1.4.4]{Mahnkopf-05} implies that the Whittaker vector $\mathcal{F}(f^{\mathcal{W}})$ evaluated at $\diag(1, t) \in T_{n-1}(F)$\footnote{The vector $\mathcal{F}(f^{\mathcal{W}})$ is nonzero only on $T^+_{n-1}(F)$. We may choose $t^{-1}_1 \in \mathcal{O}_F$ since $\pi$ admits a central character.} belongs to $\overline{\mathbb{Q}}$.
		The choice of the vector $f^{\mathcal{W}}$ implies that $f^{\mathcal{W}}(t) \in \overline{\mathbb{Q}}^{\times}$ for any element $t$ in the torus.
		Observe that we are in the following setup: There are two possibly different rational structure on $\mathrm{Ind}(\pi \times \chi)$ one coming from the Whittaker model and the other as rational structure induced from Whittaker model of inducing datum, $\mathcal{F}$ is a morphism such that $\mathcal{F}(f^{\mathcal{W}})(t) \in \overline{\mathbb{Q}}$.
		Since $\mathrm{Ind}(\pi \times \chi)$ is irreducible and admissible, $\mathcal{F}$ is a scalar. The value of this scalar is given by the ratio $\mathcal{F}(f^{\mathcal{W}})(t)/f^{\mathcal{W}}(t)$ for some $t$.
		This is an algebraic number as is argued above. Hence the lemma.
	\end{proof}

	\subsubsection{Choices for Whittaker vectors} \label{subsubsec:Whittaker-vector-choices-list}
	Equipped with this lemma and the theory of new vector of Jacquet--Piatetski-Shapiro--Shalika, we may now make the choices of the local Whittaker vectors.
	Let $S_{\chi_i}$ (resp. $S_\pi$, $S_{\psi}$) be the set of finite places of $\mathbb{Q}$ where $\chi_i$ (resp. $\pi$, $\psi$) is ramified; then put $S_{\Sigma} = S_{\pi} \cup S_{\chi_{n-2}} \cup S_{\chi_{n-1}}$. Furthermore we assume that $\chi_{n-2,v} \neq \chi_{n-1,v}$ for any $v \in  S_\Sigma$.
.
	\begin{enumerate}
		\item $v \notin S_{\Sigma}\cup \lbrace \infty \rbrace$. Let $w_{\Pi_v}$ be the new vectors in the Whittaker model of $\Pi_v$. Let $w_{\pi, v}$ be the new vector in $\pi_v$, then for $nmk \in N(\mathbb{Q}_v)M(\mathbb{Q}_v)G_{n-1}(\mathbb{Z}_v)$, define
		\[
		f^{\mathcal{W}}_{t,v}(nmk) = \rho^{t+1}_{P_{n-1}, v}(m)w^0_{\pi, v}(m) \chi_v(m).
		\]
		Let $\phi_{\Pi_v} \in \Pi_v$ be the unique element which maps $w_{\Pi_v}$ under the map $\mathcal{F}_\Pi^{\mathrm{W-aut}}$. Similarly, let $\phi_{\Sigma_{t,v}}$ to be the vector that maps under the map $\mathcal{F}_\pi^{\mathrm{W-aut}} \times id$ to $f^{\mathcal{W}}_{t, v}$ (See \eqref{eqn:comm-diagram-indWhittaker-Whittakerinduced}). When $t=0$, we denote $f^{\mathcal{W}}_{t,v}(nmk)$ simply by $f^{\mathcal{W}}_{v}(nmk)$.
			Since $w_{\pi, v}$ is a new vector, it is defined over $\overline{\mathbb{Q}}$. Therefore, it follows that $f^{\mathcal{W}}_v$ is defined over $\overline{\mathbb{Q}}$.
		
		\item $v \in S_{\Sigma}$. Let $w_{\Pi_v}$ and $w_{\Sigma_{t, v}}$ be the vectors in the Whittaker models of $\Pi_v$ and $\Sigma_v$ respectively given by the Lemma \ref{lemma:local-Whittaker-choices-ramified} for $\pi_1 = \Pi_v$, $\pi_2 = \Sigma_{t,v}$. The local Rankin--Selberg integral in this case is a nonzero algebraic number.
		Let $\phi_{\Pi_v} \in \Pi_v$ be the unique element which maps $w_{\Pi_v}$ under the map $\mathcal{F}_\Pi^{\mathrm{W-aut}}$.
		Let $\phi_{\Sigma_{t,v}}$ be the unique element in $\Sigma_{t,v}$ that is mapped to $w_{\Sigma_{t,v}}$ under the map $\mathcal{F}_v\circ (\mathcal{F}_\pi^{\mathrm{W-aut}} \times id)$. When $t=0$, we denote $w_{\Sigma_{t,v}}$ by $w_{\Sigma_v}$. Let $f^{\mathcal{W}}_v$ be the unique element in $\mathrm{Ind}(\mathcal{W}(\pi_v) \times \chi_v)$ which maps to $w_{\Sigma_v}$ under the map $\mathcal{F}_v$.
		By Lemma \ref{lemma:intertwiner-rational-whittaker-model}, $f^{\mathcal{W}}_v$ is defined over $\overline{\mathbb{Q}}$. Note that the assumption $\chi_{n-2,v} \neq \chi_{n-1,v}$ for any $v \in  S_\Sigma$
is required here to ensure that the representation $\mathrm{Ind}(\pi_v \times \chi_v)$ is
irreducible.
		
		\item $v = \lbrace \infty \rbrace$. Let $w_{\Pi_{\infty}}$ and $w_{\Sigma_{t, \infty}}$ be two nonzero vectors such that the local Rankin--Selberg integral $\Psi_{\infty}(\frac{1}{2}, w_{\Pi_{\infty}}, w_{\Sigma_{t, \infty}})$ is nonzero (See  \S \ref{subsubsec:critical-point}).
	\end{enumerate}

	\subsection{Global Whittaker vectors and Rankin--Selberg integral} \label{subsec:global-whittaker-global-integral}
		Let $\phi_{\Pi} = \otimes_v \phi_{\Pi_v}$ and $\phi_{\Sigma_t} = \otimes_v \phi_{\Sigma_{t,v}}$, where $\phi_{\Pi_v}$ and $\phi_{\Sigma_{t,v}}$ are chosen as in the previous paragraph above.
		In the paragraph below, we denote by $w_{\Sigma_t}$ the vector obtained by taking $\psi^{-1}$ Fourier coefficient of $\phi_{\Sigma_{t}} = \otimes'\phi_{\Sigma_{t,v}}$.
        The commutative diagram \eqref{eqn:comm-diagram-indWhittaker-Whittakerinduced} implies that $w_{\Sigma_t} = \otimes'w_{\Sigma_{t,v}}$. Then,
		\begin{align*}
			I\left(\frac{1}{2}, \phi_\Pi, \mathrm{Eis}(\phi_{\Sigma_t})\right) &= \underset{{U_{n-1}(\mathbb{A})\backslash \mathrm{GL}(n-1,\mathbb{A})}}{\int} w_\Pi(\iota(g)) w_{\Sigma_t}(g)dg\\
			&=\prod_v \underset{{U_{n-1}(\mathbb{Q}_v)\backslash \mathrm{GL}(n-1,\mathbb{Q}_v)}}{\int} w_{\Pi_v}(\iota(g_v)) w_{\Sigma_{t,v}}(g_v) dg_v\\
			&\equalscolon \prod_{v\in S_{\Sigma}} \Psi_v\left(\frac{1}{2}, w_{\Pi_v}, w_{\Sigma_{t,v}}\right) \cdot \prod_{v\notin S_{\Sigma}\cup \{\infty\}} \Psi_v\left(\frac{1}{2}, w_{\Pi_v}, w_{\Sigma_{t,v}}\right)\Psi_\infty\left(\frac{1}{2}, w_{\Pi_\infty}, w_{\Sigma_{t, \infty}}\right).
		\end{align*}
		For the second equality we use the fact that for $\Re(t)\gg 0$, the Fourier coefficient of Eisenstein series of pure tensors are pure tensors \cite[Prop. 7.1.3]{Shahidi-10}.
		For $v \notin S_{\Sigma}\cup \{\infty\}$, the vectors $w_{\Sigma_{t, v}}$ are a scalar multiple of spherical vectors and the scalar is given by inverse of
		\begin{equation} \label{eqn:definition-of-C}
			\begin{split}
				C_v(t) &= \underset{i<j}{\prod}L(1+2t, \pi_{i,v}\times \pi^{\vee}_{j,v}) \underset{1\leq i \leq (n-3)/2}{\prod}L(1+ (n-\tfrac{3}{2}+2i)t - \mu_{\lfloor \frac{n}{2}\rfloor} - \tfrac{1}{2}, \pi_{i,v} \times \chi^{-1}_{n-2,v})\\
				& \qquad \underset{1\leq i \leq (n-3)/2}{\prod} L(1+ (n-\tfrac{1}{2}+2i)t + \mu_{\lfloor \frac{n}{2}\rfloor} + \tfrac{1}{2}, \pi_{i,v} \times \chi^{-1}_{n-1,v}) L(1+t+ 2\mu_{\lfloor \frac{n}{2}\rfloor} +1, \chi_{n-2,v}\chi^{-1}_{n-1,v}).
			\end{split}
		\end{equation}
		Therefore for $v \notin S_{\Sigma} \cup \{\infty\}$, we get $$ \Psi_v\left(\frac{1}{2}, w_{\Pi_v}, w_{\Sigma_{t,v}}\right) = \frac{L(\frac{1}{2}, \Pi_v, \Sigma_{t,v})}{C_v(t)}.$$
		Now for $\Re(t)\gg 0$ and for any subset $S$, let $C^S(t) = \prod_{v\notin S \cup \{\infty\}} C_v(t)$, and $C(t) = C^{\lbrace \emptyset\rbrace}(t)$. Put $C^S = C^S(0)$ and $C = C(0)$. Putting all this together, we get for $\Re(t)\gg 0$
		\begin{align*}
			I\left(\frac{1}{2}, \phi_\Pi, \mathrm{Eis}(\phi_{\Sigma_t})\right) &= \frac{L^{S_\Sigma}(\frac{1}{2}, \Pi\times \Sigma_{t})\prod_{v\in S_{\Sigma}\cup \{\infty\}} \Psi_v\left(\frac{1}{2}, w_{\Pi_v}, w_{\Sigma_{t,v}}\right)}{C^{S_\Sigma}(t)}.
		\end{align*}
		Note that both the sides of the equation are well defined for $t=0$, and therefore we get the above equality holds for $\Sigma$
		\[
		I\left(\frac{1}{2}, \phi_\Pi, \mathrm{Eis}(\phi_{\Sigma})\right) = \frac{L^{S_\Sigma}(\frac{1}{2}, \Pi\times \Sigma_{})\prod_{v\in S_{\Sigma}\cup \{\infty\}} \Psi_v\left(\frac{1}{2}, w_{\Pi_v}, w_{\Sigma_{v}}\right)}{C^{S_\Sigma}}.
		\]
		Since the local $L$-function evaluated at $\frac{1}{2}$ is a nonzero algebraic number (due to generalities) and local Whittaker vectors at ramified places were chosen to ensure that the local Rankin--Selberg integral at ramified places are nonzero algebraic number(See \S \ref{subsec:choice-of-local-Whittaker-vectors-local-integral}) we may rewrite the above as
		\begin{equation} \label{eqn:integral-as-L-function}
			I\left(\frac{1}{2}, \phi_\Pi, \mathrm{Eis}(\phi_{\Sigma})\right) \sim_{\overline{\mathbb{Q}}^{\times}} \frac{L(\frac{1}{2}, \Pi\times \Sigma)\Psi_\infty\left(\frac{1}{2}, w_{\Pi_\infty}, w_{\Sigma_{\infty}}\right)}{C},
		\end{equation}
		where the symbol $\sim_{\overline{\mathbb{Q}}^{\times}}$ denotes the equality up to a nonzero algebraic number.

	\section{Proof of the Theorem \ref{thm:main-theorem}}

	\subsection{The cohomology classes}
	Define the cohomology class attached to the global Whittaker vector $w_{\Pi_f}$ as follows, 
	\begin{equation}
		\vartheta_{\Pi} \colonequals p(\Pi_f)^{-1}\mathcal{F}_{\Pi_f}(w_{\Pi_f}). 
	\end{equation}
	Since the map $p(\Pi_f)^{-1}\mathcal{F}_{\Pi_f}$ is $\mathrm{Aut}(\mathbb{C}/\overline{\mathbb{Q}})$-equivariant (See \S \ref{subsec:rat-Whittaker-Betti-periods}) and $w_{\Pi_f} \in \mathcal{W}(\Pi_f, \psi)$ is defined over $\overline{\mathbb{Q}}$ due to the choice mentioned above, we get $\vartheta_{\Pi} \in \mathrm{H}^{b_n}(S_n, E_{\mu, \overline{\mathbb{Q}}})$.
	By construction $\vartheta_{\Pi} \in \mathrm{H}^{b_n}_{cusp}(S_n, E_{\mu}) \subset \mathrm{H}^{b_n}(S_n, E_{\mu})$ and cuspidal cohomology is naturally a subspace of the cohomology with compact support, we get $\vartheta_{\Pi} \in \mathrm{H}^{b_n}_{\text{c}}(S_n, E_{\mu, \overline{\mathbb{Q}}})$.\\

	\noindent
	Let $f^{\mathcal{W}} = \otimes f^{\mathcal{W}}_v \in \mathrm{Ind}(\mathcal{W}(\pi) \times \chi)$ be the vector defined in \S \ref{subsec:choice-of-local-Whittaker-vectors-local-integral}.
	The vector $f$ is defined over $\overline{\mathbb{Q}}$ since $f^{\mathcal{W}}_{v}$ are defined over $\overline{\mathbb{Q}}$ (See \S \ref{subsubsec:Whittaker-vector-choices-list}(1),(2)).
	Recall from \S \ref{subsec:rat-Whittaker-Betti-periods} that the map
	\[
	\tilde{p}(\Sigma_f)^{-1}\mathcal{F}_{\Sigma_f} : \mathrm{Ind}(\mathcal{W}(\pi) \times \chi) \longrightarrow  \mathrm{H}^{b_{n-1}}(S_{n-1}, E_{\lambda}),
	\]
	is $\text{Aut}(\mathbb{C}/\overline{\mathbb{Q}})$-equivariant.
	Define the class
	\begin{equation}
		\vartheta_{\Sigma} \colonequals \tilde{p}(\Sigma_f)^{-1}\mathcal{F}_{\Sigma_f}(f^{\mathcal{W}}),
	\end{equation}
	where $f^{\mathcal{W}}_v$ is defined in \S \ref{subsec:choice-of-local-Whittaker-vectors-local-integral}. By construction $\vartheta_{\Sigma} \in \mathrm{H}^{b_{n-1}}(S_{n-1}, E_{\lambda, \overline{\mathbb{Q}}})$.
	Pulling back by the canonical map $p$ (See \S \ref{subsec:locally-symmetric-spaces-and-cohomology}) we get a class in $\mathrm{H}^{b_{n-1}}(\tilde{S}_{n-1}(R_f), E_{\lambda, \overline{\mathbb{Q}}})$ to be denoted by $\vartheta_{\Sigma}$.

	\subsection{Poincar\'{e} pairing of chosen Whittaker vectors} \label{subsec:Poincare-pairing-Whittaker-vector}
	Define a map $\mathcal{W}(\Pi_f) \times \mathrm{Ind}(\mathcal{W}(\pi) \times \chi) \xrightarrow{\wedge} \mathbb{C}$ as a composition of the various arrows in the diagram below so that the resulting diagram commutes.
	\[
	\begin{tikzcd}[column sep = 3ex]
		\mathcal{W}(\Pi_f) \times \mathrm{Ind}(\mathcal{W}(\pi) \times \chi) \arrow[rr, "\wedge"]  \arrow[d, hook, "p(\Pi_f)^{-1}\mathcal{F}_{\Pi_f} \times \tilde{p}(\Sigma_f)^{-1}\mathcal{F}_{\Sigma_f}"'] && \mathbb{C}\\
		\mathrm{H}^{b_n}_{\text{cusp}}(S_n, E_{\mu}) \times \mathrm{H}^{b_{n-1}}(S_{n-1}, E_{\lambda}) \arrow[d, hook] && \mathrm{H}^{t_{n-1}}_{c}(\tilde{S}_{n-1}, \mathbb{C}) \arrow[u, "\int"']\\
		\mathrm{H}^{b_n}_{c}(S_n, E_{\mu}) \times \mathrm{H}^{b_{n-1}}(S_{n-1}, E_{\lambda}) \arrow[rr, "\iota^* \times p^*"] && \mathrm{H}^{b_n}_{c}(\tilde{S}_{n-1}, E_{\mu}|_{\widetilde{S}_{n-1}}) \times \mathrm{H}^{b_{n-1}}(\tilde{S}_{n-1}, E_{\lambda}) \arrow[u, "\cup"].\\
	\end{tikzcd}
	\]
	Note that all the arrows except $\wedge$ in the diagram above is $\mathrm{Aut}(\mathbb{C}/\overline{\mathbb{Q}})$-equivariant, hence so is the map $\wedge$.
	The differential form $\vartheta_{\Pi}$ and $\vartheta_{\Sigma}$ is defined over $\overline{\mathbb{Q}}$, therefore
	\[
	\vartheta_{\Pi} \wedge \vartheta_{\Sigma} \in \im \left( \int (\mathrm{H}^{t_{n-1}}_{\text{c}}(\tilde{S}_{n-1}(R_f), \overline{\mathbb{Q}}) )\right) \cong \overline{\mathbb{Q}}.
	\]

		\noindent The construction of the differential forms above, the definition of the period integral \eqref{eqn:period-integral} coupled with the equality established in \eqref{eqn:integral-as-L-function} implies the equality below up to an algebraic number,
		\begin{equation} \label{eqn:main-theorem-identity}
			\vartheta_{\Pi}\wedge \vartheta_{\Sigma} \sim_{\overline{\mathbb{Q}}} \frac{1}{p(\Pi_f)\tilde{p}(\Sigma_f)} \frac{L(\frac{1}{2}, \Pi\times \Sigma)}{C}           \cdot p([\Pi_{\infty}], [\Sigma_{\infty}]^{\epsilon}).
		\end{equation}
	Here $p([\Pi_{\infty}], [\Sigma_{\infty}]^{\epsilon})$ denotes the complex number obtained by evaluating $\vartheta_{\Pi} \wedge \vartheta_{\Sigma}$ at the archimedean place. Using the description of the action of $\pi_0(K_{n-1, \infty})$ on orientation classes (See \S \ref{subsubsec:orientation-class-int}) and on the class $[\Pi_{\infty}] \wedge [\Sigma_{\infty}]^{\epsilon}$ (See \S \ref{subsubsec:critical-point}), we get $p([\Pi_{\infty}], [\Sigma_{\infty}]^{\epsilon})$ is invariant under $\pi_0(K_{n-1, \infty})$.
	Now, a result of Sun \cite[Theorem A]{Sun-nonvanishing} implies that $p([\Pi_{\infty}], [\Sigma_{\infty}]^{\epsilon}) \neq 0$.
	With $p(\Sigma_f) = \tilde{p}(\Sigma_f)\cdot C$, from equation \eqref{eqn:main-theorem-identity} we obtain
		\[
		L\left(\frac{1}{2}, \Pi \times \Sigma\right) \sim_{\,\overline{\mathbb{Q}}} p(\Pi_f)p(\Sigma_f)p([\Pi_{\infty}], [\Sigma_{\infty}]^{\epsilon})^{-1}.
		\]
		This completes the proof of Theorem \ref{thm:main-theorem}.

	\section{Proof of the Theorem \ref{thm:rationality-ratio-L-functions-GL(n)}}
Before we begin, we recall for convenience that $p(\Sigma_f) = \tilde{p}(\Sigma_f)\cdot C$, where $C$ is defined in \eqref{eqn:definition-of-C}. Plugging in the expression for $\tilde{p}(\Sigma_f)$ obtained from its definition we get that
		\begin{equation} \label{eqn:expression-for-period-of-Sigma}
			p(\Sigma_f) = p^{\mathrm{aut-B}}(\pi, \Sigma_{\infty})p^{\mathrm{W-aut}}(\pi_f)C.
		\end{equation}
	We will denote $\Sigma$ by $\Sigma(\pi, \chi_{n-2}, \chi_{n-1}),$ and the constant $p(\Sigma_f)$ by $p(\Sigma(\pi, \chi_{n-2}, \chi_{n-1}))$ in order to differentiate them for two different representations for different choices of characters in the definition of the representation.    
    Let $\pi_1, \dots, \pi_{(n-3)/2}$ be distinct cuspidal automorphic representation of $G_2$ such that $\pi_{i\infty} = D_{2\mu_i+n-2i}$, and $\Pi, \chi, \chi'$ be given as in the statement of the theorem. Grunwald--Wang theorem asserts that it is possible to choose an auxiliary Hecke character $\eta$ of finite-order such that $\eta_{\infty} = \epsilon(\Pi_{\infty}) \mathrm{sgn}$, $\eta$ is unramified outside of $S_{\pi}\cup S_{\chi} \cup S_{\chi'}$ and $\eta_v \notin \lbrace \chi_v, \chi'_v \rbrace$ for $v \notin S_{\pi}\cup S_{\chi} \cup S_{\chi'}$.
	We first consider the case $\mu_{\lfloor n/2 \rfloor} \geq 1$. There is the following equality,
	\[
	\frac{L(s, \Pi \times \Sigma(\pi, \chi', \eta))}{L(s, \Pi \times \Sigma(\pi, \chi, \eta))} = \frac{L(s+ \mu_{\lfloor n/2 \rfloor} + \frac{1}{2}, \Pi \times \chi')}{L(s+ \mu_{\lfloor n/2 \rfloor} + \frac{1}{2}, \Pi \times \chi)}.
	\]
	For $s=\frac{1}{2}$, Theorem \ref{thm:main-theorem} applied to the above equation implies
	\begin{equation} \label{eqn:ration-L-GL(n)}
		\frac{L(\mu_{\lfloor n/2 \rfloor} + 1, \Pi \times \chi')}{L(\mu_{\lfloor n/2 \rfloor} + 1, \Pi \times \chi)} \sim_{\overline{\mathbb{Q}}} \frac{p(\Sigma(\pi, \chi', \eta))}{p(\Sigma(\pi, \chi, \eta))}.
	\end{equation}
	Since $\pi$ remains unchanged, and the choices of $\chi,\chi'$ are made such that $\Sigma(\pi, \chi, \eta)_{\infty} = \Sigma(\pi, \chi', \eta)_{\infty}$, simplifying the above equation using \eqref{eqn:expression-for-period-of-Sigma} we get
	\[
	\frac{L(\mu_{\lfloor n/2 \rfloor} + 1, \Pi \times \chi')}{L(\mu_{\lfloor n/2 \rfloor} + 1, \Pi \times \chi)} \sim_{\overline{\mathbb{Q}}} \frac{\prod_iL(1-(\mu_{\lfloor n/2 \rfloor} + \frac{1}{2}), \pi_i \times \chi^{-1})}{\prod_iL(1-(\mu_{\lfloor n/2 \rfloor} + \frac{1}{2}), \pi_i \times \chi'^{-1})} \cdot \frac{L(2\mu_{\lfloor n/2 \rfloor}+2, \chi \times \eta^{-1})}{L(2\mu_{\lfloor n/2 \rfloor}+2, \chi' \times \eta^{-1})}.
	\]
	Note that $\frac{1}{2} - \mu_{\lfloor n/2 \rfloor}$ is a critical point for the $L$-functions $L(s, \pi_i \times \eta)$ for any finite-order Hecke character $\eta$ and $1 \leq i \leq (n-3)/2$. Since the rationality of the ratios of the $\mathrm{GL}(2)$ and $\mathrm{GL}(1)$ $L$-functions at all critical points is known from the work of Shimura, the proof of the theorem is complete in this case.\\
	
	\noindent
	Let us now assume that $n=5$, then $(n-3)/2 = 1$ and we call $\pi_1$ to be $\pi$. The only case that remains is to consider $\mu_{\lfloor n/2 \rfloor} = 0$.
	Let $\eta'$ be a finite-order Hecke character such that $\eta'_{\infty} = \mathrm{sgn}$, and $\eta'$ is unramified outside of $S_{\pi}$ and $L(\frac{1}{2}, \pi \times \eta') \neq 0$. Let $\eta'' = \eta'\chi$ and define $\tilde{\pi} \colonequals \pi \otimes \eta''$.
	Since we are allowed to choose $\pi$, with $\pi = \tilde{\pi}$ we may assume without loss of generality that $\pi$ is cohomological and $L(\frac{1}{2}, \pi \times \chi^{-1}) \neq 0$.
	Grunwald--Wang theorem asserts that it is possible to choose a finite-order Hecke character $\psi''$ such that $\psi''_{\infty} = \mathrm{sgn}$, whereas $\chi_{v} \neq \psi''_v$, $\chi'_{v} \neq \psi''_v$ for any $v \in S_{\pi} \cup S_{\chi} \cup S_{\chi'}$ and $\psi''_v$ is unramified for $v \notin S_{\pi} \cup S_{\chi} \cup S_{\chi'}$.
	A theorem of Rohrlich \cite[Theorem]{Rohrlich-89} implies that there exists $\psi'$ a character of finite-order which is unramified on $S_{\pi} \cup S_{\chi} \cup S_{\chi'} \cup \lbrace \infty \rbrace$ and $\psi'_{\infty}$ is trivial which furthermore satisfies $L(\frac{1}{2}, \pi \times \psi''\psi') \neq 0$.
	Let $\psi = \psi'\psi''$, then $\psi$ satisfies the following : $\psi_{\infty} = \mathrm{sgn}$, $\psi_v \neq \chi_v, \chi'_v$ for any $v \in S_{\pi} \cup S_{\chi} \cup S_{\chi'}$. These conditions on the characters ensure that the results from previous sections in particular Lemma \ref{lemma:intertwiner-rational-whittaker-model} and Theorem \ref{thm:main-theorem} can be applied.
	Similarly choose $\pi'$ which is a twist of $\pi$ by a Hecke character of finite-order which is trivial at infinity and unramified outside $S_{\pi}$, and another finite-order Hecke character $\psi'$ such that $L(\frac{1}{2}, \pi' \otimes \chi')L(\frac{1}{2}, \pi' \times \psi') \neq 0$. Consider the ratio,
	\[
	\frac{L(s, \Pi \times \Sigma(\pi', \chi', \psi'))}{L(s, \Pi \times \Sigma(\pi, \chi, \psi))} = \frac{L(s+\frac{1}{2}, \Pi \times \chi')}{L(s+\frac{1}{2}, \Pi \times \chi)}.
	\]
	For $s=\frac{1}{2}$, Theorem \ref{thm:main-theorem} applied to the above equation implies
	\begin{equation} \label{eqn:ration-L-GL(1)}
		\frac{L(1,\Pi \times \chi')}{L(1,\Pi \times \chi)} \sim_{\overline{\mathbb{Q}}} \frac{p(\Sigma(\pi', \chi', \psi'))}{p(\Sigma(\pi, \chi, \psi))}.
	\end{equation}
	Observe that $p^{\mathrm{aut-B}}(\pi, \Sigma_{\infty})$ and $p^{\mathrm{W-aut}}(\pi_f)$ is invariant under twist by a finite-order Hecke character which is trivial at infinity since these numbers arise when comparison of rational structures over $\overline{\mathbb{Q}}$. Using equation \eqref{eqn:expression-for-period-of-Sigma} coupled with this observation and proceeding as before we get
	\[
	\frac{L(1,\Pi \times \chi')}{L(1,\Pi \times \chi)} \sim_{\overline{\mathbb{Q}}} \frac{L(\frac{1}{2}, \pi' \times \chi'^{-1})L(2, \chi'^{-1}\psi')}{L(\frac{1}{2}, \pi \times \chi^{-1})L(2, \chi^{-1} \psi)}.
	\]
	Since the rationality of the ratios of the $\mathrm{GL}(2)$ and $\mathrm{GL}(1)$ $L$-functions is known, the proof of the theorem for $n = 5$ is complete.

\end{document}